\documentclass[10pt]{amsart}
\usepackage[T1]{fontenc}
\usepackage[latin1]{inputenc}
\usepackage{amsfonts}
\usepackage{amssymb}
\usepackage{amsmath}
\usepackage{amsthm}
\usepackage{mathtools}
\usepackage{latexsym}
\usepackage{color}
\usepackage{amscd}
\usepackage[all,cmtip]{xy}
\usepackage{graphicx}
\usepackage{url}
\usepackage[margin=1in]{geometry}
\usepackage{graphicx}
\usepackage{esint}
\usepackage[usenames,dvipsnames,svgnames,table]{xcolor}
\usepackage[colorlinks=true,
pdfstartview=FitV,linkcolor=ForestGreen,citecolor=ForestGreen,
urlcolor=blue]{hyperref}
\allowdisplaybreaks

\numberwithin{equation}{section}
\newtheorem{theorem}{Theorem}
\newtheorem{proposition}[theorem]{Proposition}
\newtheorem{lemma}[theorem]{Lemma}

\theoremstyle{definition}
\newtheorem{definition}[theorem]{Definition}

\theoremstyle{remark}
\newtheorem{remark}[theorem]{Remark}
\newtheorem{remarks}[theorem]{Remarks}

\newcommand{\norm}[1]{\left\| #1 \right\|}
\newcommand\N{{\mathbb N}}
\newcommand\R{{\mathbb R}}

\newcommand{\cC}{\mathcal C}
\newcommand{\cQ}{\mathcal Q}

\newcommand{\cU}{\mathcal{U}}
\newcommand{\cK}{\mathcal K}
\newcommand{\cT}{\mathcal T}
\newcommand{\dd}{{\, \mathrm d}}

\newcommand{\var}{\varepsilon}

\newcommand{\jesscor}[1]{\textcolor{red}}

\title[Quantitative De Giorgi Methods in Kinetic
Theory]{Quantitative De Giorgi Methods in Kinetic Theory}

\author{Jessica Guerand}

\address[Jessica Guerand]{Université de Montpellier, IMAG, 499-554 Rue du Truel, 34090 Montpellier, France}
\email{jessica.guerand@umontpellier.fr}

\author{Cl\'ement Mouhot}

\address[Cl\'ement Mouhot]{University of Cambridge, Department of
  Pure Mathematics and Mathematical Statistics, Wilberforce Road,
  Cambridge CB3 0WA, United Kingdom}

\email{c.mouhot@dpmms.cam.ac.uk}
  
\date{\today}

\begin{document}

\begin{abstract}
  We consider hypoelliptic equations of kinetic Fokker-Planck
  type, also known as Kolmogorov or ultraparabolic equations,
  with rough coefficients in the drift-diffusion operator. We
  give novel short quantitative proofs of the De Giorgi
  intermediate-value Lemma as well as weak Harnack and Harnack
  inequalities. This implies H\"{o}lder continuity with quantitative
  estimates. The paper is self-contained.
\end{abstract}

\maketitle

\setcounter{tocdepth}{1}
\tableofcontents

\section{Introduction}

\label{sec:intro}

\subsection{The problem studied}
\label{sec:theequation}

This paper is concerned with local regularity properties, namely
boundedness, Harnack inequalities and H\"{o}lder continuity, of
solutions $f=f(t,x,v)$ to the following class of hypoelliptic
partial differential equations in divergence form
\begin{equation}
  \label{e:main}
  \partial_t f + v \cdot \nabla_x f = \nabla_v \cdot (
  A \nabla_v f) + B \cdot \nabla_v f + S,
  \quad t \in \R, \ x \in \R^d, \ v \in \R^d
\end{equation} 
where $A=A(t,x,v)$, $B=B(t,x,v)$ and $S=S(t,x,v)$ satisfy (for some
constants $0 < \lambda < \Lambda$):
\begin{equation}
  \label{e:hyp-coef}
  \begin{dcases} 
    \text{$A$ is a measurable symmetric real matrix field with
      eigenvalues in } [\lambda,\Lambda],\\
    \text{$B$ is a measurable vector field such that }
    |B| \le \Lambda,\\
    \text{$S$ is a real scalar field in $L^\infty$.}
  \end{dcases}
\end{equation} 
This equation naturally appears in kinetic theory where it is
refereed to as the \emph{kinetic Fokker-Planck equation}; it is
included in the class considered by Kolmogorov~\cite{kolmogorov}
(with constant $A$ and linear $B$) that inspired the theory of
hypoellipticity of H\"{o}rmander~\cite{MR0222474} (see
\cite{AP20}). The coefficients are called ``rough'' because $A$,
$B$ and $S$ in the drift-diffusion operator on the $v$ variable
are merely measurable with no further regularity.

Our class~\eqref{e:main}-\eqref{e:hyp-coef} is invariant under
translations in $t$, $x$ and under \emph{Galilean translations},
i.e. under $z \to z_0 \circ z$ where $z_0 =(t_0,x_0,v_0)$,
$z = (t,x,v)$ and with the non-commutative group operation
\begin{equation*}
  z_0 \circ z = (t_0+t,x_0+x + t v_0, v_0 + v).
\end{equation*}
Finally for any $r>0$ it is invariant under the scaling
$z=(t,x,v) \to rz := (r^2t,r^3x,r v)$. Using the invariances of
the equation, we define for $z_0 \in \R^{1+2d}$ and $r>0$:
\begin{align*}
  Q_r(z_0) := z_0 \circ \left[ r Q_1 \right]  
  := \left\{ -r^2 <t-t_0 \le 0, \ |x - x_0- (t-t_0)v_0|
  < r^3, \ |v-v_0| < r \right\}
\end{align*}
and we simply write $Q_r(0)=Q_r$ when $z_0=0$. We denote $|E|$
the Lebesgue measure of a Lebesgue set $E$. We write
$a \lesssim b$ (resp.  $a \gtrsim b$) when $a \le Cb$ (resp.
$a \ge Cb$) for some constant $C>0$ whose only relevant
dependency, if any, is specified in the index, as in
$\lesssim_{\text{parameter}}$. We write $a\sim b$ if
$a \lesssim b$ and $a \gtrsim b$. We write $\fint$ for integrals
normalized by the volume of the integration domain, and
$\cT:=\partial_t+v\cdot \nabla_x$.

\begin{definition}[Weak solution, sub-solution, super-solution]
  \label{d:weak}
  Let $\cU = (a,b) \times \Omega_x \times \Omega_v$ with
  $-\infty<a < b \le +\infty$ and $\Omega_x$ and $\Omega_v$ two
  open sets of $\R^d$. A function $f: \cU \to \R$ is a \emph{weak
    solution} of \eqref{e:main} on $\cU$ if
  $f \in L^\infty ((a,b);L^2 (\Omega_x \times \Omega_v)) \cap L^2
  ((a,b) \times \Omega_x;H^1(\Omega_v))$ and \eqref{e:main} is
  satisfied in the sense of distributions in $\cU$. A function
  $f$ is a \emph{weak sub-solution} of \eqref{e:main} if
  $f \in L^\infty ((a,b);L^2 (\Omega_x \times \Omega_v)) \cap L^2
  ((a,b) \times \Omega_x;H^1(\Omega_v))$ and for all
  $\beta : \R \to \R$ in $C^2$ with $\beta' \ge 0$ and
  $\beta'' \ge 0$ both bounded, and any non-negative
  $\varphi \in C^\infty_c (\cU)$,
  \begin{equation*} 
    - \int_\cU \beta(f) \cT \varphi \dd z
    \le - \int_{\cU} A \nabla_v \beta(f) \cdot \nabla_v \varphi \dd
    z + \int_{\cU} \left[ B \cdot \nabla_v \beta(f)+S \beta'(f)
    \right] \varphi \dd z.
  \end{equation*}
  It is a \emph{weak super-solution} of \eqref{e:main} if $-f$ is
  a weak sub-solution.
\end{definition}

\begin{remark}
  \label{r:def gen}
  This definition is equivalent to those in \cite{pp} and
  \cite{gimv} in the case of solutions, but is weaker than them
  in the case of sub- and super-solutions. Indeed~\cite{pp,gimv}
  make respectively the extra regularity assumption
  $\cT f \in L^{2}((a,b) \times \Omega_x \times \Omega_v)$ or
  $\cT f \in L^2 ((a,b) \times \Omega_x;H^{-1}(\Omega_v))$.
  These assumptions were introduced to justify the energy
  estimates. It is however enough to assume the renormalization
  formulation above, and it allows to include important
  sub-solutions such as for instance $f=f(t)={\bf 1}_{t \le 0}$
  (when $S=0$) which were excluded by the definition
  in~\cite{pp,gimv}. Our definition is equivalent to that of De
  Giorgi in the elliptic case (and reminiscent of the definition
  of solutions in \cite{GV15}).
\end{remark}

\subsection{Main contributions}
\label{sec:contributions}

Given the invariances, we only state results in unit
centered cylinders.
\begin{figure}[h]
  \includegraphics[width=12cm]{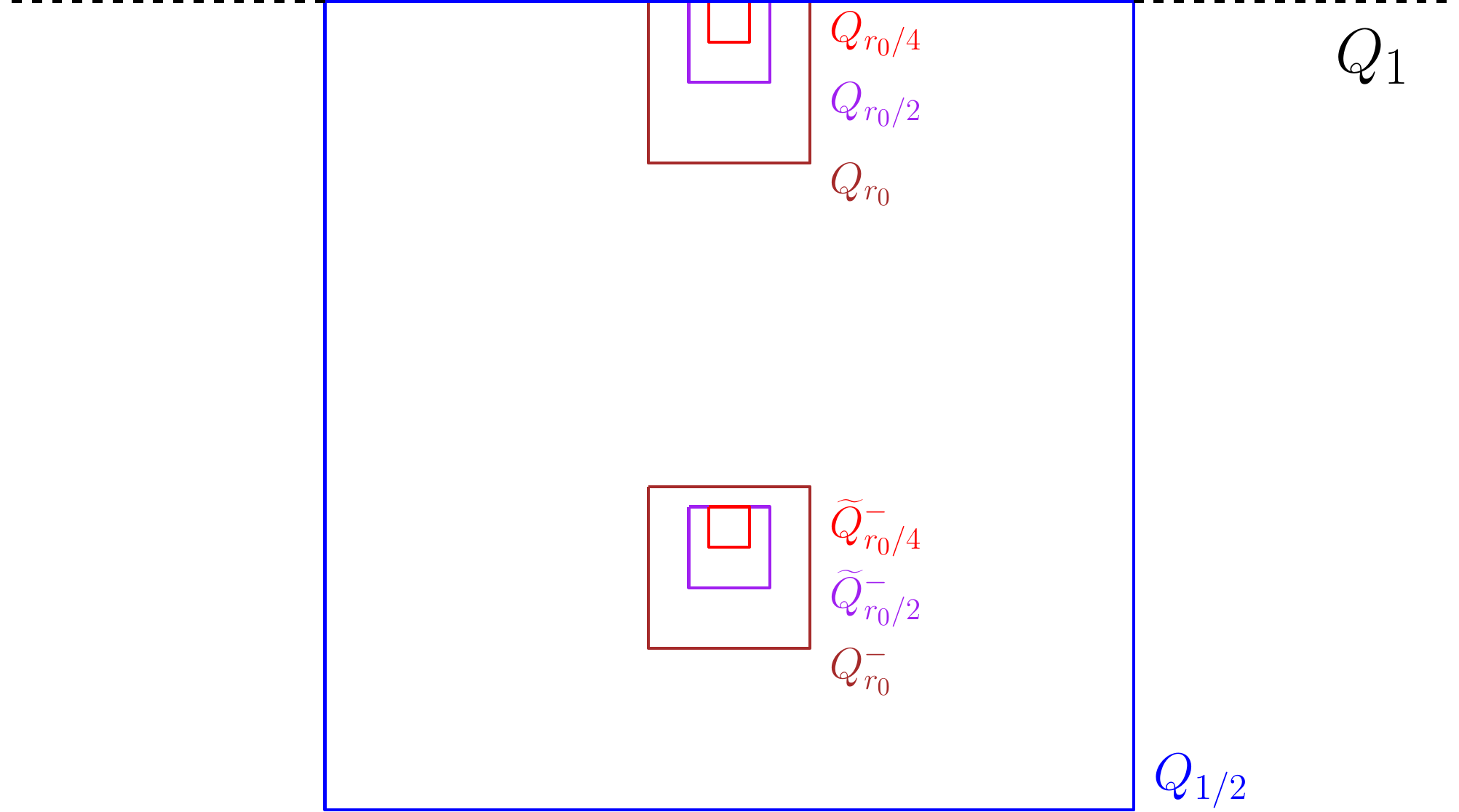}
  \caption{The different cylinders in the Intermediate Value
    Lemma and Harnack inequalities.}
  \label{fig:LVI}
\end{figure}
\begin{theorem}[Intermediate Value Lemma]
  \label{t:IVL}
  Given $\delta_1,\delta_2 \in (0,1)$, there are explicit
  constants
  $r_0 \sim
  \frac{\sqrt{\delta_1}}{\sqrt{1+\|S\|_{L^\infty(Q_1)}}}$ in
  $(0,\frac{1}{20})$ if $S \not =0$ and $r_0=\frac{1}{20}$ if
  $S=0$, and
  $\theta \sim
  \frac{(\delta_1\delta_2)^{10d+15}}{(1+\|S\|_{L^\infty(Q_1)})^{4d+3}}$
  and
  $\nu \gtrsim \frac{\left( \delta_1 \delta_2
    \right)^{10d+16}}{(1+\|S\|_{L^\infty(Q_1)})^{2d+1}}$ both in
  $(0,1)$, such that any sub-solution
  $f:Q_1\rightarrow \mathbb{R}$ to
  \eqref{e:main}-\eqref{e:hyp-coef} so that $f\leq 1$ in
  $Q_{\frac{1}{2}}$ and
  \begin{equation}
    \label{e: hyp IVL}
    |\{f\leq 0\}\cap Q_{r_0}^{-}|   \geq  \delta_1
    |Q_{r_0}^{-}|
    \quad \text{ and } \quad 
    |\{f\geq 1-\theta\} \cap Q_{r_0}|   \geq  \delta_2
    |Q_{r_0}|
  \end{equation}
  i.e. we control the measure of where $f$ is below $0$ and above
  $(1-\theta)$, satisfies
  \begin{equation}
    \label{e: concl IVL}
    \left|\left\{ 0<f< 1-\theta \right\}\cap Q_{\frac12} \right|
    \geq \nu |Q_{\frac12}|,
  \end{equation}
  where
  $Q_{r_0}^{-} := Q_{r_0}(-2r_0^2,0,0) =(-3r_0^2,-2r_0^2]
  \times B_{r_0^3} \times B_{r_0}$ (see Figure~\ref{fig:LVI}).
\end{theorem}

\begin{remark}
  \label{rem:gap}
  This lemma is the kinetic quantitative counterpart of the
  quantitative elliptic \cite{DG56,DG57,vasseur} and parabolic
  \cite{gueDGhalv2} intermediate value lemma.  As in the
  parabolic case, past and a future cylinders $Q_{r_0}^{-}$ and
  $Q_{r_0}^{+}$ are required to be disjoint but contrary to the
  parabolic case, a gap in time between the two cylinders is also
  required. This gap is also mentioned in~\cite{gimv,AP20}. Let
  us explain why it cannot be removed. Consider for instance
  $S=0$ and velocities bounded by $|v| \le V_m$ in the
  cylinder. Then ${\bf 1}_{x+ct<a}$ is a sub-solution for any
  $a\in \mathbb{R}$ and $|c|>V_m$. If $Q_{r_0}^{-}$ and $Q_{r_0}$
  were too close, a line of discontinuity of the form $x+ct=a$
  could cross both and the previous sub-solution would contradict
  the conclusion of Theorem~\ref{t:IVL}.
\end{remark}

\begin{theorem}[Harnack inequalities]
  \label{t:harnack}
  There is $\zeta >0$ depending only $\lambda,\Lambda$ such that
  any non-negative weak super-solution $f$
  to~\eqref{e:main}-\eqref{e:hyp-coef} in $Q_1$ satisfies the
  weak Harnack inequality
  \begin{equation}
    \label{eq:w-Harnack-stat}
    \left( \int_{\tilde Q_{\frac{r_0}{2}} ^-} f^\zeta (z) \dd t \dd x \dd v 
    \right)^{\frac{1}{\zeta}}
    \lesssim_{\lambda,\Lambda} \inf_{ Q_{\frac{r_0}{2}}}f 
    + \|S\|_{L^\infty (Q_1)}
  \end{equation} 
  where $r_0=\frac{1}{20}$ and
  $\tilde Q_{\frac{r_0}{2}}
  ^{-}:=Q_{\frac{r_0}{2}}((-\frac{19}{8}r_0^2,0,0))$ (see
  Figure~\ref{fig:LVI}), and any non-negative weak solution $f$
  to~\eqref{e:main}-\eqref{e:hyp-coef} satisfies the following
  Harnack inequality (with
  $\tilde
  Q_{\frac{r_0}{4}}^{-}:=Q_{\frac{r_0}{4}}((-\frac{19}{8}r_0^2,0,0))$)
  \begin{equation}
    \label{eq:s-Harnack-stat}
    \sup_{\tilde Q_{\frac{r_0}{4}} ^{-}} f \lesssim_{\lambda,\Lambda}
    \inf_{Q_{\frac{r_0}{4}}} f + \| S \|_{L^\infty(Q_1)}.
  \end{equation}
\end{theorem}

\begin{remarks}
\begin{enumerate}
\item The ``weak'' Harnack inequality, in spite of its name, is
  not weaker than Harnack inequality since it holds for
  super-solutions. Combined with the $L^{\zeta} \to L^\infty$
  gain of integrability in~Proposition~\ref{prop:1st lemma}, it
  implies the Harnack inequality for solutions. Super-solutions
  of the form ${\bf 1}_{x+ct\ge a}$ for $a\in \mathbb{R}$ and
  $|c|>V_m$ (included in our definition) show that the gap in
  time is required in~\eqref{eq:w-Harnack-stat}.
\item The Harnack inequality for equation~\eqref{e:main} was
  first proved in~\cite{gimv} by a non-constructive argument. The
  present paper provides a new constructive De Giorgi
  approach. Another constructive proof by the Moser-Kru\v{z}kov
  approach is proposed in~\cite{Guerand-Imbert}. The weak Harnack
  inequality was obtained for the long-range Boltzmann equation
  in~\cite{MR4049224}, and was proved for the kinetic
  Fokker-Planck equations considered in this paper
  in~\cite{Guerand-Imbert} by the Moser-Kru\v{z}kov approach.
\item As compared to that in~\cite{Guerand-Imbert}, our approach
  is based on trajectorial arguments and does not require working
  on the logarithm of the solution or the so-called inkspot
  lemma. Our Poincar\'e inequality (Proposition~\ref{p:HPI
    lemma}) and measure-to-pointwise estimate (Lemma
  \ref{l:increase}) take into account a gap in time which removes
  the requirement for the sub-solution to be considered in a
  large domain. Our Poincar\'e inequality also holds without an
  information in measure around the center of the cylinder as
  in~\cite{Guerand-Imbert}.
\end{enumerate}
\end{remarks}

\begin{theorem}[H\"{o}lder Continuity]
  \label{t:holder}
  There is $\alpha \in (0,1)$, computable from the proof and only
  depending on $\lambda, \Lambda$ and $\|S\|_{L^{\infty}}$, such
  that any weak solution $f$ of \eqref{e:main}-\eqref{e:hyp-coef}
  in $Q_2$ satisfies
  \begin{equation*}
    [f]_{C^\alpha (Q_{1})} := \sup_{z_1,z_2 \in Q_1, \ z_1 \not =
      z_2} \frac{|f(z_1) - f(z_2)|}{|z_1-z_2|^{\alpha}}
    \lesssim_{\lambda,\Lambda}
    \left(1+\|S\|_{L^\infty(Q_2)}\right)
    \left( \|f\|_{L^2 (Q_2)} + \|S\|_{L^\infty(Q_2)} \right).
  \end{equation*}
\end{theorem}

\begin{remark} 
  %\label{rmk: IVL}
  The H\"{o}lder continuity was first proved in \cite{wz09,wz11}
  (with constructive method) and this proof is revisited and
  simplified in~\cite{Guerand-Imbert}, including ideas and
  methods from~\cite{moser,kru63,kru64}. An alternative
  non-constructive proof was proposed in~\cite{gimv} following
  the De Giorgi method \cite{DG56}: the non-constructive part was
  the intermediate-value lemma and we provide here a new
  constructive argument.
\end{remark}

\subsection{Structure of the method}

The core of our proof is, given $f$ sub-solution
to~\eqref{e:main}-\eqref{e:hyp-coef} with $S=0$:
\begin{align*}
  f \in L^\zeta, \ \zeta >0
  & \quad \xrightarrow{(1)}
    \quad f \in  L^\infty \cap L^1_{t,v}W^{\frac13-0,1}_x 
    \quad  \xrightarrow{(2)} \quad
    \text{Weak Poincar\'e inequality in $L^1$} \\
  & \quad \xrightarrow{(3)}
    \quad \text{Intermediate Value Lemma (Theorem~\ref{t:IVL}})
    \quad \xrightarrow{(4)}
    \quad \text{Measure-to-pointwise estimate} \\
  & \quad \xrightarrow{(5)} \quad \text{Weak log-Harnack
    estimate} \quad \xrightarrow{(6)}
    \quad \text{Weak Harnack estimate.}
\end{align*}
Once these steps are proved, it is immediate to prove the Harnack
inequality for solutions by combining the weak Harnack inequality
for super-solutions and step~(1) for sub-solutions. The H\"{o}lder
continuity follows classically (see Subsection~\ref{ss:Holder})
from either the measure-to-pointwise estimate applied to both
sub-solutions $f$ and $-f$, or from the Harnack inequality.
Step~(1) (Section~\ref{sec:int}) is semi-novel: it elaborates
upon ideas in~\cite{pp} to prove the first Lemma of De Giorgi as
well as a gain of Sobolev regularity with the help of Kolmogorov
fundamental solutions. Step~(2) (Proposition~\ref{p:HPI lemma})
is the most novel step and introduces an argument based on
trajectories and the previous Sobolev regularity to ``noise'' the
$x$-dependency of the trajectories. Step~(3) (proof in
Subsection~\ref{ss:proof-ivl}) is novel and based on simple
energy estimates. Step~(4) (Lemma~\ref{l:increase} in
Subsection~\ref{ss:m2p}) is standard and sketched for the sake of
obtaining quantitative constants. Step~(5) (in
Section~\ref{sec:Harnack}) is semi-novel but immediate when
constants are quantified properly in the previous steps. Step~(6)
(in Section~\ref{sec:Harnack}) is novel in the context of
hypoelliptic equations but inspired from elliptic
equations~\cite{MR3565366}; it uses an induction, Vitali's
covering lemma and Step~(5) at every scale.

\section{Integral estimates revisited}
\label{sec:int}

In this section, we briefly revisit estimates from~\cite{pp,gimv}
on the gain of integrability for sub-solutions (the kinetic
counterpart to the first lemma of De Giorgi) and the low-order
Sobolev regularity estimate for sub-solutions, first mentioned
in~\cite{gimv}. We provide new proofs based on fundamental
solutions which, albeit variants of existing ones, seem simpler
and optimal.

\subsection{The energy estimate}

\begin{proposition}[Energy estimate]
  \label{prop:EE}
  Let $f$ be a non-negative weak sub-solution
  to~\eqref{e:main}-\eqref{e:hyp-coef} in an open set
  $\cU \in \R^{1+2d}$. Given any
  $Q_r(z_0) \subset Q_R(z_0) \subset \cU$ with $0<r<R$, one has
  \begin{align*}
    \int_{Q_r(z_0)} |\nabla_v f|^2
    \lesssim_{\lambda,\Lambda} \cC(r,R,v_0) \left( 
    \int_{Q_R(z_0)} f^2 + \int_{Q_R(z_0)}  f |S| \right)
  \end{align*}
  where $z_0=(t_0,x_0,v_0)$,
  $Q_r^\tau (z_0) = \{(x,v)\in\R^{2d} \, : \, (\tau,x,v)\in
  Q_r(z_0) \} $, and
  \begin{equation}
    \label{eq:C}
    \cC(r,R,v_0) := \left( 1 + \frac{1}{(R-r)^2} +
    \frac{|v_0|+R}{(R-r)r^2} +\frac{1}{(R-r)r} \right).
  \end{equation}
\end{proposition}

\begin{proof}[Proof of Proposition~\ref{prop:EE}]
  Consider $\varphi$ a smooth function valued in $[0,1]$ that is
  equal to $1$ on $Q_r(z_0)$ and $0$ outside $Q_R(z_0)$. In order
  to use $f\varphi^2$ as a test function, we argue by
  density. Introduce
  \begin{equation*}
    \psi_n * [f \varphi] \varphi (z) := 
    \int_{z'\in\R^{2d+1}} \psi_n \left(t-t',x-tv-(x'-t'v'),v-v'
    \right) f(z') \varphi(z') \varphi(z)
  \end{equation*}
  where $\psi_n(t,x,v)=n^{4d+2} \psi(n^2 t,n^3 x, n v)$ and
  $\psi(t,x,v):= \pi^{-d-\frac{1}{2}} e^{-t^2-|x|^2-|v|^2}$.
  Then
  \begin{align*}
    I_n:=  \big\langle \cT f, \psi_n * [f \varphi] \varphi 
    \big\rangle
    =  \big\langle f, \psi_n * [(\cT f)  \varphi] \varphi 
    \big\rangle 
    & = \frac{1}{2} \left[\big\langle \cT f, \psi_n * [f \varphi] \varphi 
      \big\rangle + \big\langle f, \psi_n * [(\cT f)  \varphi] \varphi 
      \big\rangle \right] \\
    & = \frac{1}{2} \left[-\big\langle f, \psi_n * [f \varphi] (\cT\varphi) 
      \big\rangle -\big\langle f, \psi_n * [f (\cT \varphi)] \varphi 
      \big\rangle  \right]
  \end{align*}
  which converges to
  $-\frac{1}{2} \big\langle f^2, \cT\varphi^2 \big\rangle$ as $n
  \to \infty$. The
  other terms in the inequation converge thanks to the bound
  $f \in L^\infty ((a,b);L^2 (\Omega_x \times \Omega_v)) \cap L^2
  ((a,b) \times \Omega_x;H^1(\Omega_v))$. We deduce
  \begin{align*}
    \lambda \int_{Q_R(z_0)} |\nabla_v f|^2
    \varphi^2 \dd z 
    & \le \int_{Q_R(z_0)} f^2 \Big( |\partial_t \varphi|
      \varphi + (|v_0|+R) | \nabla_x \varphi| \varphi \Big) \dd z
      + \Lambda \int_{Q_R(z_0)} |\nabla_v
      f| |\nabla_v \varphi| f \varphi \dd z \\
    & \qquad + \Lambda
      \int_{Q_R(z_0)} |\nabla_v f| f \varphi^2 \dd z +
      \int_{Q_R(z_0)} f |S| \varphi^2 \dd z.
  \end{align*}
  The result follows from Cauchy-Schwarz' inequality
  and $|\partial_t \varphi| \lesssim \frac{1}{(R-r)r}$,
  $|\nabla_x \varphi| \lesssim \frac{1}{(R-r)r^2}$, 
  $|\nabla_v \varphi| \lesssim \frac{1}{(R-r)}$.
\end{proof}

\subsection{Integral estimates on Kolmogorov fundamental
  solutions}

We denote $\cK := \cT - \Delta_v$.
\begin{lemma}[Estimates on the fundamental solution with constant
  coefficients]
  \label{lem:Kolmogorov}
  Consider $f \ge 0$ locally integrable so that
  $\cK f = \nabla_v \cdot F_1 + F_2 - m$ with
  $F_1,F_2 \in L^1 \cap L^2(\R_- \times \R^{2d})$ and
  $0 \le m \in M^1(\R_- \times \R^{2d})$ (a non-negative measure
  with finite mass) and where $F_1, F_2$ and $m$ have compact
  support in time included in some $\left(-\tau,0\right]$. Then
  there for any $p \in [2,2+\frac1d)$ and
  $\sigma \in [0,\frac13)$
  \begin{align}
    \label{eq:Kolm1}
    & \| f\|_{L^{p}(\R_- \times \R^{2d})} \lesssim_{\tau,\lambda,\Lambda}
      \left(2+\frac1d-p\right)^{-1}
      \left[ \left\| F_1 \right\|_{L^2(\R_- \times \R^{2d})} +
      \left\| F_2 \right\|_{L^2(\R_- \times \R^{2d})} \right] \\
    \label{eq:Kolm2}
    & \| f\|_{L^1_{t,v} W^{\sigma,1}_x(\R_- \times \R^{2d})}
      \lesssim_{\tau,\lambda,\Lambda}
      \left(\frac13-\sigma \right)^{-1} \left[ \|
      F_1 \|_{L^1(\R_- \times \R^{2d})} + \left\| F_2
      \right\|_{L^1(\R_- \times \R^{2d})} + \| m \|_{M^1(\R_- \times
      \R^{2d})} \right].
  \end{align}
\end{lemma}

\begin{proof}[Proof of Lemma~\ref{lem:Kolmogorov}]
We use the fundamental solution computed by Kolmogorov
  in~\cite{kolmogorov} (see for
  instance~\cite[Appendix~A]{MR4113786} for details):
  \begin{align*}
    & \forall \, t \in \R_-, \ x,v \in \R^d, \quad
      f(t,x,v) =  \int_{(t',x',v')\in \R^{2d+1}} G(t-t', x-x'
      -(t-t')v',v-v') (\cK f)(t',x',v') \\
    &    \forall \, t \ge 0, \ x,v \in \R^d, \quad
      G(t,x,v) :=
      \begin{dcases}
        \left( \frac{3}{4\pi^2 t^4} \right)^{\frac{d}{2}}
        \exp \left[ - \frac{3 \left| x - \frac{t}{2} v
        \right|^2}{t^3} - \frac{|v|^2}{4t} \right] & \mbox{ if } t>0,\\
        0 &  \mbox{ if } t\leq 0.
      \end{dcases}
  \end{align*}
  Since $f$ and $G$ are non-negative, we deduce that
  \begin{align*}
    0 \le f(t,x,v) \le \int_{(t',x',v')\in \R^{2d+1}} 
    G(t-t', x-x'-(t-t')v',v-v')
    \Big[(\nabla_{v'} \cdot F_1) (t',x',v')+ F_2 (t',x',v') \Big]
  \end{align*}
  and since
  \begin{align*}
    \forall \, t \ge 0, \ x,v \in \R^d, \quad 
    |\nabla_v G(t,x,v)| +t  |\nabla_x G(t,x,v)|
    \lesssim t^{-2d-\frac12}
    \exp \left[ - \frac{3 \left| x - \frac{t}{2} v
    \right|^2}{2t^3} - \frac{|v|^2}{8t} \right]
  \end{align*}
  we have
  $\nabla_v G, t \nabla_x G \in L^{\frac{2d+1}{2d+1/2}-0}
  (\left(0,\tau\right) \times \R^{2d})$ and therefore by
  integration by parts
  \begin{align*}
    f(t,x,v) \le \int_{(t',x',v')\in \R^{2d+1}} 
    \nabla_{v'}  G(t-t',x-x'-(t-t')v',v-v') F_1(t',x',v')\\
    + \int_{(t',x',v') \in \R^{2d+1}} 
    G(t-t',x-x'-(t-t')v',v-v') F_2(t',x',v')
  \end{align*}
  and Young's convolution inequality (which works in unimodular
  spaces like $(\R^{2d+1},\circ)$ with the Lebesgue measure), we
  deduce, by tracking down the dependency in $p$ of the constant
  \begin{align*}
    \forall \, p \in \left[2,2+\frac1d\right), \quad
    \| f \|_{L^{p}(\R_- \times \R^{2d})} \lesssim_\tau
    \left( 2+ \frac1d - p \right)^{-1} \left[ \| F_1 \|_{L^2(\R_-
    \times \R^{2d})}
    + \left\| F_2 \right\|_{L^2(\R_- \times \R^{2d})} \right].
  \end{align*}
  (The threshold $2+\frac1d$ is likely to be optimal.) This
  proves~\eqref{eq:Kolm1}. To prove~\eqref{eq:Kolm2} split
  \begin{align*}
    G = G_\var + G_\var ^\bot \quad \text{ with } \quad
    G_\var(t,x,v) := \chi \left( \frac{t}{\var} \right) G(t,x,v)
  \end{align*}
  where $\var >0$ and $\chi$ is a smooth function on $\R_+$
  valued in $[0,1]$ equal to $1$ in $[0,1]$ and $0$ on
  $[2,+\infty)$. We have the following simple estimates for every
  $l \in \mathbb{N}$
  \begin{align*}
    & \left|\nabla_x^l G_\var ^\bot(t,x,v)\right| \lesssim_l
      \var^{-\frac32 l}t^{-2d}
      \exp \left[ - \frac{3 \left| x - \frac{t}{2} v
      \right|^2}{2t^3} - \frac{|v|^2}{8t} \right] \\
    & \left|\nabla_v \nabla_x^l G_\var ^\bot(t,x,v)\right|
      + t\left|\nabla_x \nabla_x^l G_\var ^\bot(t,x,v)\right|
      \lesssim_l \var^{-\frac32 l -\frac12} t^{-2d}
      \exp \left[ - \frac{3 \left| x - \frac{t}{2} v
      \right|^2}{2t^3} - \frac{|v|^2}{8t} \right]
  \end{align*}
  which straightforwardly implies (assuming $\tau \ge 1$ and
  $\var <1$ wlog)
  \begin{align*}
    & \left\| G_\var^\bot \right\|_{L^1_{t,v}((0,\tau)
      \times \R^d; W^{l,1}_x(\R^d)))} +
      \left\| \nabla_v G_\var ^\bot \right\|_{L^1_{t,v}((0,\tau)
      \times \R^d ; W^{l,1}_x(\R^d))}
      + \left\|t\nabla_x G_\var^\bot \right\|_{L^1_{t,v}((0,\tau)
      \times \R^d ; W^{l,1}_x(\R^d))}
      \lesssim_l \tau \var^{-\frac32 l -\frac12} \\
    & \| G_\var \|_{L^1((0,\tau) \times \R^{2d})} + \left\|
      \nabla_v G_\var \right\|_{L^1((0,\tau) \times \R^{2d})} 
      +  \left\|t\nabla_x G_\var \right\|_{L^1((0,\tau) \times \R^{2d})}
      \lesssim \tau \var^{\frac12}.  
  \end{align*}
  The splitting $G=G_\var + G_\var ^\bot$ yields
  $f = f_\var + f_\var ^\bot$, and the convolution inequality
  $M^1 * L^1 \to L^1$ implies
  \begin{align*}
    & \| f_\var \|_{L^1(\R_- \times \R^{2d})} \lesssim
      \tau \var^{\frac12} \left(  \| F_1 \|_{L^1(\R_- \times \R^{2d})} +
      \left\| F_2 \right\|_{L^1(\R_- \times \R^{2d})}
      + \| m \|_{M^1(\R_- \times
      \R^{2d})} \right) \\
    & \left\| f_\var ^\bot \right\|_{L^1_{t,v}W^{l,1}_x(\R_-
      \times \R^{2d})} \lesssim
      \tau \var^{-\frac32 l -\frac12}
      \left[ \| F_1 \|_{L^1(\R_- \times \R^{2d})} +
      \left\| F_2 \right\|_{L^1(\R_- \times \R^{2d})} 
      + \| m \|_{M^1(\R_- \times \R^{2d})} \right].
  \end{align*}
  Since this decomposition holds for all $\var >0$, it implies by
  standard interpolation the estimate~\eqref{eq:Kolm2} for any
  $\sigma \in [0,\frac{1}{3})$ (again the exponent is likely to
  be optimal but in any case our constant degenerates as
  $\sigma \to \frac{1}{3}$). In order to be self-contained let us
  a give a short proof. Given $\sigma \in [0,\frac13)$, we
  Fourier-transform and decompose dyadically, defining
  $\langle \xi \rangle:= (1+|\xi|^2)^{\frac{1}{2}}$
  \begin{align}
    \nonumber
    (1-\Delta_x)^{\frac{\sigma}{2}} f(t,x,v)
    & = \int_{\xi,y \in \R^d} e^{i \xi
      \cdot (x-y)} \langle \xi \rangle^\sigma
      f(t,y,v) 
     = \sum_{k \ge -1} \int_{\xi,y \in \R^d} e^{i \xi
      \cdot (x-y)} a_k(\xi) f(t,y,v) \\
    \label{eq:decomp}
    & = \sum_{k \ge -1} \int_{\xi,y \in \R^d} e^{i \xi
      \cdot (x-y)} B_k(\xi) (1-\Delta_y)^{\frac{l}{2}} f(t,y,v)
  \end{align}
  where $a_k(\xi) := \langle \xi \rangle^\sigma \varphi_k$ and
  $B_k(\xi) := \langle \xi \rangle^{\sigma-l} \varphi_k$, and
  where we have defined in the standard way
  $\varphi_k(\xi) := [ \chi(2^{-k}\xi) - \chi(2^{-k+1} \xi) ]$
  for $k \ge 0$ with $\chi$ a smooth function valued in $[0,1]$
  and equal to $1$ in $B(0,1)$ and $0$ outside $B(0,2)$, and
  $\varphi_{-1}(\xi) = \sum_{k \le -1} [ \chi(2^{-k}\xi) -
  \chi(2^{-k+1} \xi) ]$. For a given $F=F(y)$ one has
  \begin{align*}
    & \int_{x \in \R^d} \left| \int_{\xi,y \in \R^d} e^{i \xi
      \cdot (x-y)} a_k(\xi) F(y)
      \right| \lesssim 2^{k\sigma} \| F \|_{L^1} \\
    & \int_{x \in \R^d} \left| \int_{\xi,y \in \R^d} e^{i \xi
      \cdot (x-y)} B_k(\xi) (1-\Delta_y) F(y)
      \right| \lesssim 2^{k(\sigma-l)} \| F \|_{W^{l,1}}
  \end{align*}
  by splitting the integrand into $|x-y| \le 2^{-k}$ and
  $|x-y| > 2^{-k}$ and integrating by parts the operator
  $\Delta_\xi^{\frac{\ell}{2}}$ with $\ell$ even and strictly
  greater than $d$. We then use the
  decomposition~\eqref{eq:decomp} in the ``$a_k$'' form on
  $f_\var$ and in the ``$B_k$'' form on $f_\var ^\bot$, and with
  a $\var=\var_k$ depending on $k$ defined below:
  \begin{align*}
    & \left\| (1-\Delta_x)^{\frac{\sigma}{2}} f \right\|_{L^1(\R_-
    \times \R^{2d})} \\
    & \lesssim \tau \sum_{k \ge -1} \left( \var_k
      ^{\frac12} 2^{k\sigma} + \var_k ^{-\frac{3}{2}l -\frac12}
      2^{k(\sigma-l)} \right)
      \left[ \| F_1 \|_{L^1(\R_- \times \R^{2d})}
      + \left\| F_2 \right\|_{L^1(\R_- \times \R^{2d})} 
      + \| m \|_{M^1(\R_- \times \R^{2d})} \right]\\
    & \lesssim
      \frac{\tau}{\delta}  \left[ \| F_1 \|_{L^1(\R_- \times \R^{2d})}
      + \left\| F_2 \right\|_{L^1(\R_- \times \R^{2d})}
      + \| m \|_{M^1(\R_- \times \R^{2d})} \right]
  \end{align*}
  with the choice $\sigma = \frac13 - \delta \in [0,\frac13)$ and
  $\var_k := 2^{-2k(\frac13-\frac{\delta}{2})}$ and
  $l> 1+\frac{4}{9\delta}$. This concludes the proof.
\end{proof}

\subsection{Integral estimates for sub-solutions}

We combine the previous lemma with a localization argument and
the energy estimate to get the
\begin{proposition}[Integral regularization estimates for
  non-negative sub-solutions]
  \label{prop:Gain sol}
  Let $f$ be a non-negative weak sub-solution
  to~\eqref{e:main}-\eqref{e:hyp-coef} in an open set
  $\cU \in \R^{1+2d}$. Given any
  $Q_r(z_0) \subset Q_R(z_0) \subset \cU$ with $0<r<R \le 1$, and
  any $p \in [2,2+\frac1d)$ and $\sigma \in [0,\frac13)$, $f$
  satisfies
  \begin{align}
    \label{eq:gain-int}
    \|f\|_{L^{p}(Q_r(z_0))} \lesssim \left( 2+ \frac1d - p\right)^{-1}
    \cC'(r,R,v_0) \left[ \norm{f}_{L^{2}(Q_R(z_0))}
    + \|S\|_{L^2(Q_R(z_0))} \right] \\
    \label{eq:gain-reg}
    \| f \|_{L^1_{t,v}W^{\sigma,1}_x(Q_r(z_0))} \lesssim \left(
    \frac13 - \sigma \right)^{-1}
    \cC''(r,R,v_0) \left[ \norm{f}_{L^{2}(Q_R(z_0))}
    + \|S\|_{L^2(Q_R(z_0))} \right]
  \end{align}
  where $\cC$ was defined in~\eqref{eq:C} and 
  \begin{align*}
    \cC'(r,R,v_0) := \left( 1+ \frac{1}{R-r} \right) \cC(r,R,v_0)
    \quad \text{ and } \quad 
    \cC''(r,R,v_0) := R^{1+2d} \left( 1+ \frac{1}{R-r} \right)
    \cC(r,R,v_0).
  \end{align*}
\end{proposition}

\begin{proof}[Proof of Proposition~\ref{prop:Gain sol}]
  Since $f$ is a sub-solution to~\eqref{e:main}, there is a
  non-negative measure $\bar m \ge 0$ so that
  \begin{align*}
    \cT f =\nabla_v\cdot (A\nabla_v f) + B \cdot \nabla_v f
    + S - \bar m.
  \end{align*}
  Consider $\varphi_1$ smooth valued in $[0,1]$ and equal to $1$
  on $Q_r(z_0)$ and $0$ outside $Q_{r+\frac{R-r}{2}}(z_0)$ and
  $g_1:=\varphi_1 f$. The latter satisfies
  \begin{equation}
    \label{eq:ineqKg}
    \cK g_1 = \nabla_v \cdot F_1+F_2 - m \quad \text{ with }
      \begin{dcases}
        m:= \bar m \varphi_1, \\
        F_1 := (A\nabla_v f) \varphi_1-(\nabla_v
        f) \varphi_1 - f \nabla_v \varphi_1, \\
        F_2:= -A\nabla_v f\cdot \nabla_v \varphi_1+ \left( B
          \cdot \nabla_v f \right)\varphi_1 + S\varphi_1 + f \cT
        \varphi_1.
      \end{dcases}
  \end{equation}
  The energy estimate in Proposition \ref{prop:EE} implies
  \begin{align*}
    \| F_1 \|_{L^2(\R_- \times \R^{2d})} + \| F_2 \|_{L^2(\R_- \times \R^{2d})} 
    & \lesssim \left( 1+\frac{1}{R-r} \right)
      \cC\left(r+\frac{R-r}{2}, R,v_0\right)
      \left( \| f  \|_{L^2\left(Q_R(z_0)\right)} 
      + \| S \|_{L^2(Q_R(z_0))}  \right) \\
    & \lesssim \cC'(r,R,z_0) \left( \| f
      \|_{L^2(Q_R(z_0))}  + \| S \|_{L^2(Q_R(z_0))}  \right)
  \end{align*}
  which, combined with~\eqref{eq:Kolm1},
  shows~\eqref{eq:gain-int}.
  
  Consider then $\varphi_2$ smooth valued in $[0,1]$ and equal to
  $1$ on $Q_{r+\frac{R-r}{2}}(z_0)$ and $0$ outside $Q_R(z_0)$
  and $g_2:=\varphi_2 f$. The function $g_2$ satisfies a similar
  equation as $g_1$ in \eqref{eq:ineqKg}, with $\varphi_2$
  replacing $\varphi_1$. Integrating this equation simply against
  $1$ yields (thanks to the cancellation of divergence terms)
  \begin{align*}
    \| \bar m \|_{M^1\left(Q_{r+\frac{R-r}{2}}(z_0)\right)}
    & \lesssim 
      \| \varphi_2 m \|_{M^1(\R_- \times \R^{2d})} \lesssim
      \int_{Q_{r+\frac{R-r}{2}}(z_0)} \left[ -A\nabla_v f\cdot
      \nabla_v \varphi_2 + \left( B \cdot \nabla_v f
      \right)\varphi_2 + S\varphi_2 + f \cT \varphi_2 \right] \\
    & \lesssim \cC\left(r+\frac{R-r}{2},R,v_0 \right) \| f
      \|_{L^1(Q_R(z_0))} + \| S \|_{L^1(Q_R(z_0))}\\
    & \lesssim \cC\left(r,R,v_0
      \right)\left[ \| f \|_{L^2(Q_R(z_0))} + \| S
      \|_{L^2(Q_R(z_0))} \right].
  \end{align*}
  Combined with~\eqref{eq:Kolm2} and (thanks to the localization)
  \begin{align*}
    \| F_1 \|_{L^1(\R_- \times \R^{2d})} +    \| F_2 \|_{L^1(\R_-
    \times \R^{2d})}  \lesssim  \| F_1 \|_{L^2(\R_- \times
    \R^{2d})} + \| F_2 \|_{L^2(\R_- \times \R^{2d})},
  \end{align*}
  it implies~\eqref{eq:gain-reg}.
\end{proof}

\subsection{Iterated gain of integrability for sub-solutions}
 
We give a short proof of this result first obtained in
\cite[Theorem~1.2]{pp} and then proved
differently~\cite[Theorem~12]{gimv}. This is the counterpart of
the ``first lemma of De Giorgi'' for elliptic equations, in the
context of kinetic hypoelliptic equations. We allow for an
initial integrability $L^\zeta$ with exponent $\zeta \in (0,2)$
(such extension is well-known for elliptic equations).

\begin{proposition}[Upper bound for sub-solutions]
  \label{prop:1st lemma}
  Let $f$ be a non-negative weak sub-solution
  to~\eqref{e:main}-\eqref{e:hyp-coef} in an open set
  $\cU \in \R^{1+2d}$. Given any
  $Q_r(z_0) \subset Q_R(z_0) \subset \cU$ with $0<r<R \le 1$, and
  $\zeta >0$, $f$ satisfies
  \begin{equation*}
    \| f \|_{L^\infty(Q_r(z_0))} \lesssim_{\lambda,\Lambda}
    \left( \frac{1+|v_0|}{r^2(R-r)^3}
    \right)^{\frac{1+4d}{\zeta}}
    \left[ \|f\|_{L^\zeta(Q_R(z_0))} + \|S\|_{L^\infty(Q_R(z_0))}\right].
  \end{equation*}
\end{proposition}

\begin{proof}[Proof of Proposition~\ref{prop:1st lemma}]
  Fix $p_0:=2+\frac{1}{2d}$ and define $q:=\frac{p_0}{2}$ and
  $q_n := q^n$. Consider $\beta_{n,k}$ on $\R_+$ with
  $\beta_{n,k}' \ge 0$ and $\beta''_{n,k} \ge 0$ both bounded and
  so that $\beta_{n,k}(z) \to z^{q_n}$ as $k \to \infty$ and
  $\beta_{n,k}(z) \lesssim z^{q_n}$ and
  $\beta_{n,k}'(z) \lesssim z^{q_n-1}$ uniformly in $k \in
  \N^*$. The Definition~\ref{d:weak} implies that
  $\beta_{n,k}(f)$ is a weak sub-solution with source term
  $S_{n,k} := \beta_{n,k}'(f) S$. Define $r_0=R$ and
  $r_n := r_{n-1} - \delta n^{-2}$ with
  $\delta = \frac12 \left( \sum_{k \ge 1} k^{-2}
  \right)^{-1}(R-r)$. Since $p_0 \in [2,2+\frac1d)$, the
  estimate~\eqref{eq:gain-int} implies for all $n \ge 1$
  \begin{align*}
    \|\beta_{n,k}(f)\|_{L^{p_0}(Q_{r_n}(z_0))}  
    & \lesssim
      \cC'(r_n,r_{n-1},v_0) \left[
      \|\beta_{n,k}(f)\|_{L^{2}(Q_{r_{n-1}}(z_0))}
      + \|S_{n,k}\|_{L^2(Q_{r_{n-1}}(z_0))} \right] \\ 
    & \lesssim  \frac{(1+|v_0|)n^6}{r^2(R-r)^3} 
      \left[ \|\beta_{n,k}(f)\|_{L^{2}(Q_{r_{n-1}}(z_0))}
      + \|S_{n,k}\|_{L^2(Q_{r_{n-1}}(z_0))} \right]
  \end{align*}
  for $n \ge 1$, which means by taking $k \to \infty$ and coming
  back to $f$
  \begin{align*}
    \|f\|_{L^{2q_{n+1}}(Q_{r_n}(z_0))}  
    & \lesssim \left( \frac{(1+|v_0|)n^6}{r^2(R-r)^3}
      \right)^{\frac{1}{q^n}} 
      2^{-1+\frac{1}{q_n}} \left[
      \|f\|_{L^{2q_n}(Q_{r_{n-1}}(z_0))} + \|
      f\|_{L^{2q_n}(Q_{r_{n-1}}(z_0))}^{1-\frac{1}{q_n}}
      \|S\|_{L^\infty(Q_R(z_0))}^{\frac{1}{q_n}} \right] \\
    & \lesssim \left( \frac{(1+|v_0|)n^6}{r^2(R-r)^3}
      \right)^{\frac{1}{q^n}} 
      \left[ \left( 1 + \frac{1}{q_n} \right) 
      \|f\|_{L^{2q_n}(Q_{r_{n-1}}(z_0))} + \frac{1}{q_n} 
      \|S\|_{L^\infty(Q_R(z_0))} \right],
  \end{align*}
  assuming by induction
  $\|f\|_{L^{2q_n}(Q_{r_{n-1}}(z_0))} < +\infty$. The convergence
  of the infinite product then implies
  \begin{equation*}
    \|f\|_{L^\infty(Q_r(z_0))}  
    \lesssim \left( \frac{1+|v_0|}{r^2(R-r)^3}
    \right)^{1+4d} \left[ \|f\|_{L^2(Q_R(z_0))} +
      \|S\|_{L^\infty(Q_R(z_0))}\right].
  \end{equation*}
  This proves the claim when $\zeta \ge 2$. To prove it when
  $\zeta \in (0,2)$, we deduce from the previous estimate
  \begin{align*}
    \|f\|_{L^\infty(Q_r(z_0))}  + \|S\|_{L^\infty(Q_r(z_0))}
    \lesssim \left( \frac{1+|v_0|}{r^2(R-r)^3}
    \right)^{1+4d} \left[ \|f\|_{L^\infty(Q_R(z_0))} ^{1-\zeta}
    \|f\|_{L^\zeta(Q_R(z_0))} ^\zeta +
    \|S\|_{L^\infty(Q_R(z_0))}\right]
  \end{align*}
  and thus by Young inequality, the quantity $A(r) :=
  \|f\|_{L^\infty(Q_r(z_0))}  + \|S\|_{L^\infty(Q_r(z_0))}$ satisfies,
  for some $C>0$,
  \begin{align*}
    A(r) \le \frac12 A(R) + C \left( \frac{1+|v_0|}{r^2(R-r)^3}
    \right)^{\frac{1+4d}{\zeta}} \left[ \|f\|_{L^\zeta(Q_R(z_0))}  +
    \|S\|_{L^\infty(Q_R(z_0))}\right].
  \end{align*}
  Introducing an (increasing this time) sequence of radii
  $r_n := r_{n-1} + \delta n^{-2}$ we obtain by induction
  \begin{align*}
    & A(r_n) \le \frac12 A(r_{n+1}) + C n^{\frac{2+8d}{\zeta}}
    \left( \frac{1+|v_0|}{r^2(R-r)^3}
    \right)^{\frac{1+4d}{\zeta}} \left[ \|f\|_{L^\zeta(Q_R(z_0))}  +
    \|S\|_{L^\infty(Q_R(z_0))}\right] \\
    & A(r_0) \le \left( \frac12 \right)^n A(r_{n+1}) +
    C \left( \sum_{k=1} ^n \frac{k^{\frac{2+8d}{\zeta}}}{2^k} \right)
    \left( \frac{1+|v_0|}{r^2(R-r)^3}
    \right)^{\frac{1+4d}{\zeta}} \left[ \|f\|_{L^\zeta(Q_R(z_0))}  +
    \|S\|_{L^\infty(Q_R(z_0))}\right]
  \end{align*}
  which yields the result by taking $n \to \infty$ in the right
  hand side.
\end{proof}

\section{Intermediate-Value Lemma and oscillations}

\subsection{Weak Poincar\'e inequality} 

The adjective `weak' refers to the small additional $L^2$ error
term below.

\begin{proposition}[Hypoelliptic Poincar\'e inequality with error]
  \label{p:HPI lemma}
  Given any $\var\in (0,1)$ and $\sigma \in (0,\frac13)$, any
  non-negative sub-solution $f$ to
  \eqref{e:main}-\eqref{e:hyp-coef} on $Q_5$ satisfies
  \begin{equation}
    \label{HPI inequality}
    \norm{\left(f-\langle f
        \rangle_{Q_1^{-}}\right)_+}_{L^{1}(Q_1)}
    \lesssim_{\lambda,\Lambda} 
    \frac{1}{\varepsilon^{d+2}} \norm{\nabla_v f}_{L^1 (Q_5)}
    + \varepsilon^\sigma  \left(
    \frac13 - \sigma \right)^{-1} \norm{f}_{L^2(Q_5)} +
    \norm{S}_{L^2(Q_5)},
  \end{equation}
  where $Q_1^- := Q_1(-2,0,0) = (-3,-2] \times B_1 \times B_1$
  and
  $\langle f \rangle_{Q_1^{-}} := \fint_{Q_1^-} f :=
  \frac{1}{|Q_1^-|} \int_{Q_1 ^-} f$.
\end{proposition}

\begin{remark}
  The motivation for the following argument was \cite[Lemma~10,
  p.11]{vasseur} where a simple quantitative proof of the
  intermediate value Lemma of De Giorgi (also sometimes called De
  Giorgi's isoperimetric inequality) is sketched in the elliptic
  case, based on introducing the trajectory between two points of
  the domain and using the vector field $\nabla_v$ to connect
  them. We have to deal here with the hypoelliptic structure.
\end{remark}

\begin{proof}
  Consider, for $\var \in (0,1)$, a smooth function
  $\varphi_\var=\varphi_\var(y,w)$ which satisfies
  $0\leq \varphi_\var \leq 1$ and has compact support in
  $B_1^{2}$ and such that $\varphi_\var=1$ in
  $B_{(1-\var)}\times B_{(1-\var)}$ and with
  $|\nabla_y \varphi_\var|\lesssim \var^{-1}$ and
  $|\nabla_w \varphi_\var|\lesssim \var^{-1}$. We then split the
  integral to be estimated as follows
  \begin{align*}
    & \norm{\left(f-\langle f
      \rangle_{Q_1^{-}}\right)_+}_{L^{1}(Q_1)}  \lesssim
      \norm{\left(f-\langle f \varphi_\var
      \rangle_{Q_1^{-}}\right)_+}_{L^{1}(Q_1)}\\
    & \lesssim
      \int_{(t,x,v) \in Q_1}
      \left\{ \fint_{(s,y,w) \in Q_1^-} \left[
      f(t,x,v)-f(s,y,w) \right] \varphi_\var(y,
      w) \right\}_+  +  \norm{f}_{L^1(Q_1)}
      \fint_{Q_{1}^{-}}\left( 1 -  \varphi_\var(y,w) \right)\\
    & \lesssim \int_{(t,x,v) \in Q_1}
      \left\{ \fint_{(s,y,w) \in Q_1^-} \left[
      f(t,x,v)-f(s,y,w) \right] \varphi_\var(y,
      w) \right\}_+  + \var^{2d} \norm{f}_{L^2(Q_1)} 
  \end{align*}
  where we have used
  $\langle f \varphi_\var \rangle_{Q_1^{-}}\leq \langle f
  \rangle_{Q_1^{-}}$ and the Cauchy-Schwarz inequality.

  Let us estimate the first term of the previous
  inequality. Given $t,x,v$ fixed, we decompose the trajectory
  $(t,x,v)\rightarrow (s,y,w)$ into four sub-trajectories in
  $Q_5$: a trajectory of length $O(\var)$ along $\nabla_x$, two
  trajectories of length $O(1)$ along $\nabla_v$, and finally one
  trajectory of length $O(1)$ along
  $\cT := \partial_t + v \cdot \nabla_x$:
  \begin{equation*}
    (t,x,v)\underset{\nabla_x}{\longrightarrow}
    (t,x+\varepsilon w,v) \underset{\nabla_v}{\longrightarrow}
    \left(t,x+\varepsilon w, \frac{x+\varepsilon w -y}{t-s}\right)
    \underset{\cT}{\longrightarrow} \left(s,
      y,\frac{x+\varepsilon w -y}{t-s}\right)
    \underset{\nabla_v}{\longrightarrow} (s,y,w).
  \end{equation*}
  The first sub-trajectory is estimated by the {\it integral}
  regularity $L^1_{t,v} W^{\sigma,1}_{x}$ proved
  in~\eqref{eq:gain-reg}. The other trajectories are estimated
  directly by the vector fields in the equation. The position
  $x+\varepsilon w \in Q_2$ since $x,w \in B_1$ and
  $\var \in (0,1)$. The velocity
  $\frac{x+\varepsilon w -y}{t-s} \in Q_3$ since $x,w,y \in B_1$
  and $t-s\geq 1$ due to the definitions of $Q_{1}^+$ and
  $Q_{1}^{-}$, and this velocity yields a transport line from
  $(t,x+\varepsilon w)$ to $(s,y)$. Note that we are implicitly
  using the H\"{o}rmander commutator condition:
  $\nabla_v,\cT, [\nabla_v,\cT]$ span all the vector fields on
  $\R^{2d+1}$.

  Decompose along the previous trajectories
  \begin{align*}
    & f(t,x,v)-f(s,y,w)
      =  \Big[ f(t,x,v)-f(t,x+\varepsilon w,v) \Big]  +
      \left[ f(t,x+\varepsilon
      w,v)- f\left(t,x+\varepsilon w,\frac{x+\varepsilon w -y}{t-s}
      \right) \right] \\
    & \qquad + \left[ f\left(t,x+\varepsilon w,\frac{x+\varepsilon w
      -y}{t-s}\right)- f\left(s,y,\frac{x+\varepsilon w
      -y}{t-s}\right) \right] 
      + \left[ f\left(s,y,\frac{x+\varepsilon w
      -y}{t-s}\right)-f(s,y,w) \right]
  \end{align*}
  and integrate against $\varphi_\var(y,w)$ on
  $(s,y,w)\in Q_{1}^-$, which gives the four terms
  \begin{align*}
    I_1(t,x,v)
    & := \int_{(s,y,w)\in Q_1^{-}}
      \Big[ f(t,x,v)-f(t,x+\varepsilon
      w,v)\Big] \varphi_\var(y,w),\\
    I_2(t,x,v)
    & := \int_{(s,y,w)\in Q_1^{-}}
      \left[ f(t,x+\varepsilon w,v)-f\left(t,x+\varepsilon
      w,\frac{x+\varepsilon w -y}{t-s}\right)\right] 
      \varphi_\var(y,w), \\
    I_3(t,x,v)
    & := \int_{(s,y,w)\in Q_1^{-}} \left[ f\left(t,x+\varepsilon
      w,\frac{x+\varepsilon w
      -y}{t-s}\right)-f\left(s,y,\frac{x+\varepsilon w
      -y}{t-s}\right)\right] \varphi_\var(y,w), \\
    I_4(t,x,v)
    & := \int_{(s,y,w)\in Q_1^{-}}
      \left[ f\left(s,y,\frac{x+\varepsilon w
      -y}{t-s}\right)-f(s,y,w)\right] \varphi_\var(y,w).
  \end{align*}

  Regarding the term $I_2$, we use Taylor's formula and
  $0\leq \varphi_\var\leq 1$ to deduce
  \begin{align*}
    I_2(t,x,v)
    &\leq \int_{(s,y,w) \in Q_1^{-}} \int_{\tau\in [0,1]}
      \left(v-\frac{x+\varepsilon w-y}{t-s}\right)\cdot \nabla_v
      f \left(t,x+\varepsilon w, \tau v +
      (1-\tau)\frac{x+\varepsilon w-y}{t-s} \right)
      \varphi_\var(y,w) \\
    &\lesssim \int_{(s,y,w) \in Q_1^{-}} \int_{\tau\in [0,1]}
      |\nabla_v f| \left(t,x+\varepsilon w, \tau v +
      (1-\tau)\frac{x+\varepsilon w-y}{t-s} \right). 
  \end{align*} 
  Integrate then on $(t,x,v)\in Q_{1}^+$ to get
  \begin{align}
    \label{estim I2}
    \nonumber
    \fint_{(t,x,v) \in Q_1} I_2
    &\lesssim \int_{(t,X,v)\in (-1,0)\times B_2 \times
      B_1} \int_{(s,Y,w)\in (-3,-2)\times B_4 \times B_1}
      \int_{\tau\in (0,1)} |\nabla_v f| \left(t,X, v+ (1-\tau) Y
      \right) \\
    & \lesssim \int_{(t,X,V)\in (-1,0)\times B_2 \times
      B_5} \int_{(s,Y,w)\in (-3,-2)\times B_4 \times B_1}
      \int_{\tau\in (0,1)} |\nabla_v f| \left(t,X, V\right)
      \lesssim \int_{Q_5}  |\nabla_v f|,
  \end{align}
  where we have used successively the following changes of
  variables with bounded Jacobians:
  \begin{equation*}
    x\rightarrow X=x+\varepsilon w \in B_2, \qquad 
    y\rightarrow Y= \frac{X -y}{t-s} -v\in B_4, \qquad
    v\rightarrow V=v+(1-\tau)Y\in B_5.
  \end{equation*}

  The term $I_4$ is treated like $I_2$:
  \begin{align}
    \label{estim I4}
    \fint_{(t,x,v) \in Q_1} I_4 \lesssim \int_{Q_5}  |\nabla_v f|,
  \end{align}

  Regarding the term $I_1 $, we perform the change of variable
  $w \in B_1 \rightarrow x'=x+\varepsilon w \in B_\var(x)$ with
  Jacobian $\var^{-d}$ and use the $L^1_{t,v} W^{\sigma,1}_{x}$
  regularity of non-negative sub-solutions proved
  in~\eqref{eq:gain-reg}:
  \begin{align}
    \label{estim I1}
    \nonumber  \fint_{(t,x,v)\in Q_1}  I_1
    & \lesssim \int_{(t,x,v)\in Q_1, \, (s,y,w)\in Q_1^{-}}
      |f(t,x,v)-f(t,x+\varepsilon w,v) | \\
    \nonumber
    &\lesssim \int_{(t,x,v)\in Q_1, \, w \in B_1}
      \frac{|f(t,x,v)-f(t,x+\varepsilon w,v)|}{|\varepsilon w
      |^{d+\sigma}}|\varepsilon w|^{d+\sigma} \\
    \nonumber
    &\lesssim \var^\sigma \int_{(t,x,v)\in Q_1, \, x' \in B_2 }
      \frac{|f(t,x,v)-f(t,x',v)|}{|x-x'|^{d+\sigma}} \\
    & \lesssim \var^\sigma \norm{f}_{L^1_{t,v}W^{\sigma,1}_x(Q_2)}
      \lesssim \var^\sigma \left( \frac13 -\sigma \right)^{-1}
      \left[ \norm{f}_{L^2(Q_3)} + \norm{S}_{L^2(Q_3)} \right]. 
  \end{align}

  Regarding the term $I_3$, we note first that
  $\cT f \in L^2_{t,x} H^{-1}_v + M^1_{t,x,v}$ with finite norm
  in $Q_R(z_0)$ (arguing as in proof of
  Proposition~\ref{prop:Gain sol}). The Taylor formula between
  $(t,x+\varepsilon w)$ and $(s,y)$ along $\cT$ thus holds in
  weak form against $\varphi_\var$ thanks to the latter bounds
  and the non-singular change of
  variable~\eqref{eq:change-variables} discussed below:
  \begin{align}
    \nonumber
    I_3(t,x,v)
    & =  \int_{(s,y,w)\in Q_1^{-}} \left[
      f\left(t,x+\varepsilon w,\frac{x+\varepsilon w
      -y}{t-s}\right)-f\left(s,y,\frac{x+\varepsilon w
      -y}{t-s}\right)\right] \varphi_\var(y,w)\\
    \label{eq:Taylor-weak}
    & \lesssim \int_{(s,y,w)\in Q_1^{-}} \int_{\tau \in
      [0,1]} (t-s)\cT
      f\left(\tau t+(1-\tau)s,\tau (x+\varepsilon
      w)+(1-\tau)y,\frac{x+\varepsilon w
      -y}{t-s}\right)\varphi_\var(y,w).
  \end{align}
  We then use the fact that $f$ is a sub-solution
  to~\eqref{e:main} in the distributional sense:
  \begin{align*}
    I_3(t,x,v)
    & \lesssim \int_{(s,y,w)\in Q_1^{-}} \int_{\tau \in
      [0,1]} (t-s)\nabla_v \cdot (A \nabla_v f)\left(\tau
      t+(1-\tau)s,\tau (x+\varepsilon
      w)+(1-\tau)y,\frac{x+\varepsilon w
      -y}{t-s}\right)\varphi_\var(y,w)\\
    & \quad +\int_{(s,y,w)\in Q_1^{-}} \int_{\tau \in
      [0,1]} (t-s)B\cdot \nabla_v f\left(\tau t+(1-\tau)s,\tau
      (x+\varepsilon w)+(1-\tau)y,\frac{x+\varepsilon w
      -y}{t-s}\right)\varphi_\var(y,w)\\
    & \quad + \int_{(s,y,w)\in Q_1^{-}} \int_{\tau \in
      [0,1]} (t-s)S\left(\tau t+(1-\tau)s,\tau (x+\varepsilon
      w)+(1-\tau)y,\frac{x+\varepsilon w
      -y}{t-s}\right)\varphi_\var(y,w)\\
    & := I_{31}+I_{32}+I_{33}. 
  \end{align*}
  Arguing as for $I_2$ and $I_4$, we have
  \begin{align}
    \label{estim I323}
    \fint_{(t,x,v) \in Q_1} I_{32} \lesssim \int_{Q_5} \left|
    \nabla_v f\right| \quad \text{ and } \quad
    \fint_{(t,x,v) \in Q_1} I_{33} \lesssim \int_{Q_5} |S|,
  \end{align}
  where we performed consecutively the changes of variable
  $y\rightarrow V=\frac{x+\varepsilon w-y}{t-s}$,
  $x\rightarrow X= x+\varepsilon w -(1-\tau)(t-s)V$,
  $s\rightarrow s'=t-s$ and $t'\rightarrow t-(1-\tau)s'$.  To
  estimate the remaining term $I_{31}$, we use the change of
  variable
  \begin{equation}
    \label{eq:change-variables}
    (y,w) \mapsto (Y,W) \quad \text{ with } \quad
    Y:=\tau (x+\varepsilon w)+(1-\tau)y \quad \text{
      and } \quad W:=\frac{x+\varepsilon w -y}{t-s} 
  \end{equation}
  such that $(y,w)\mapsto (Y,W)$ is a bijection from the set
  $(B_1)^2$ to the (diamond-shaped) set
  \begin{equation*}
    E:=E(\tau,\varepsilon,t,s,x) \subset B\left(\tau x, (1-\tau)
      + \tau \var\right)
    \times B\left(\frac{x}{t-s},\frac{1+\var}{t-s}\right) \
    \subset \ B_2 \times B_3
  \end{equation*}
  with Jacobian $(\frac{\var}{t-s})^d$ and which maps
  respective boundaries (to compute the Jacobian easily use the
  formula
  $\det ( \begin{smallmatrix} A & B \\ C & D \end{smallmatrix}) =
  \det(A - B D^{-1} C) \det D$). We deduce 
  \begin{align*}
    I_{31} = \frac{1}{\varepsilon^d} \int_{\tau\in
    [0,1]}\int_{s\in (-3,-2), \, (Y,W)\in E}
    (t-s)^{d+1}\nabla_v\cdot (A
    \nabla_v f)\left(\tau
    t+(1-\tau)s,Y,W\right)\\
    \times \varphi_{\varepsilon} \left(Y-\tau
    (t-s)W,\frac{Y-x+(1-\tau)(t-s)W}{\varepsilon} \right)
  \end{align*}
  and we integrate by parts in $W$, using that $\varphi_\var=0$
  on the boundary of $E(\tau,\var,t,s,x)$:
  \begin{align*}
    I_{31}= \
    & \frac{1}{\varepsilon^d}\int_{\tau \in
      [0,1]}\int_{s\in (-3,-2), (Y,W)\in E}  (t-s)^{d+1}(A \nabla_v
      f)\left(\tau t+(1-\tau)s,Y,W\right) \\
    & \qquad \times \bigg[\tau (t-s)\nabla_y\varphi_\var
      \left(Y - \tau (t-s)W,\frac{Y-x+(1-\tau)(t-s)W}{\varepsilon}
      \right)\\
    & \quad \qquad
      -\frac{(1-\tau)(t-s)}{\var}\nabla_w\varphi_\var
      \left(Y-\tau (t-s)W,\frac{Y-x+(1-\tau)(t-s)W}{\varepsilon}
      \right)\bigg].
  \end{align*} 
  Using the bounds on the derivatives of $\varphi_\var$ then
  yields
  \begin{align}
    I_{31}(t,x,v) \lesssim \frac{1}{ \varepsilon^{d+2}}
    \int_{\tau \in [0,1]}\int_{s\in
    (-3,-2), (Y,W)\in E}   |\nabla_v f|\left(\tau
    t+(1-\tau)s,Y,W\right) \quad \Rightarrow \quad
    \label{estim I31}
    \fint_{Q_1} I_{31} \lesssim
    \frac{1}{ \varepsilon^{d+2}} \int_{Q_3}   |\nabla_v f|.
  \end{align}
  The result follows from combining \eqref{estim I2},
  \eqref{estim I4}, \eqref{estim I1}, \eqref{estim I323}
  and~\eqref{estim I31}.
\end{proof}

\begin{remark}
  Note that the regularity $W^{\sigma,1}_{x}$ is only used over a
  small trajectory that ``noises'' the position variable $x$ in
  $Q_1$ with the velocity $w$ in $Q_1^-$, hence allowing to
  integrate by parts the diffusion operator using \emph{only} the
  variables in $Q_1^-$. Note also that it is possible to get some
  $W^{\sigma',1}_{t,x,v}$ regularity in all variable with
  $\sigma' \in (0,\sigma)$ small by the same method as in
  Lemma~\ref{lem:Kolmogorov}, however such regularity is too weak
  to yield any intermediate value estimate alone. Note also that
  the gap in time between $Q_1^-$ and $Q_1$ is used to make sure
  the intermediate velocity $\frac{x+\var w-y}{t-s}$ remains
  bounded and the various domains of integration remain bounded
  along the velocity variable. In fact, the result is false
  without such gap, see Remark~\ref{rem:gap}.
\end{remark}

\subsection{Proof of the Intermediate Value Lemma}
\label{ss:proof-ivl}

In this subsection, we prove that Proposition~\ref{p:HPI lemma}
implies Theorem~\ref{t:IVL}. Take $f$ a sub-solution
to~\eqref{e:main}-\eqref{e:hyp-coef} on $Q_1$ and
satisfying~\eqref{e: hyp IVL} for some given
$\delta_1,\delta_2>0$:
\begin{align}
  \label{hyp bis IVL}
  |\{f\leq 0\}\cap Q_{r_0}^{-}|  \geq  \delta_1 |Q_{r_0}^{-}|
  \quad \text{ and }\ \quad 
  |\{f\geq 1-\theta\} \cap Q_{r_0}|  \geq  \delta_2 |Q_{r_0}|.
\end{align}
Define $g:=f-(t+25r_0^2) \|S\|_{L^\infty(Q_1)}$.  Then its
positive part $g_+$ is a sub-solution
to~\eqref{e:main}-\eqref{e:hyp-coef} in $Q_{5r_0}$ with zero
source term and with $g_+ \in [0,1]$ since $f\le 1$ in
$Q_{\frac{1}{2}}$. We set
$r_0 = \left(\frac{\delta_1}{400
    (1+\|S\|_{L^\infty(Q_1)})}\right)^{\frac12} \leq
\frac{1}{20}$ if $S$ non-zero and $r_0=\frac{1}{20}$ if $S=0$,
and we apply~\eqref{HPI inequality} to $g_+$ at scale $r_0$, for
some $\var>0$ to be chosen later:
\begin{equation}
  \label{Poinc in proof}
  \fint_{Q_{r_0}} \left(g_+-\langle
    g_+\rangle_{Q_{r_0}^{-}}\right)_{+} \lesssim
  \frac{r_0}{\varepsilon^{d+2}} \fint_{Q_{5r_0}} |\nabla_v g_+|
  + \var^\sigma  \left( \fint_{Q_{5r_0}} g_+^2 \right)^{\frac12}
  \lesssim
  \frac{1}{r_0^{4d+1} \varepsilon^{d+2}} \int_{Q_{5r_0}} |\nabla_v g_+|
  + \var^\sigma
\end{equation}
where we have used the bound $g_+ \in [0,1]$ to control the $L^2$
norm.  Then~\eqref{hyp bis IVL} implies
\begin{equation}
  \label{estim averag}
  \langle g_+\rangle_{Q_{r_0}^{-}} = \fint_{(s,y,w) \in
    Q_{r_0}^-} \big[f(s,y,w) - (s+25r_0^2) \| S \|_{L^\infty(Q_1)}\big]_+
  \leq \frac{\left|{\{f>0\}\cap Q_{r_0}^-}\right|}{|Q^{-}_{r_0}|}
  \leq 1-\delta_1
\end{equation}
and 
\begin{align}
  \label{lower bd Poinc}
  \nonumber
  \fint_{Q_{r_0}} \left(g_+-\langle g_+
  \rangle_{Q_{r_0}^{-}}\right)_{+}
  &\ge \frac{1}{|Q_{r_0}|} \int_{(t,x,v) \in Q_{r_0}}
    \big[ f(t,x,v) - (t+25r_0^2) \|S\|_{L^\infty(Q_1)}
    -(1-\delta_1)\big]_{+}\\
  \nonumber
  &\ge \frac{1}{|Q_{r_0}|} \int_{(t,x,v) \in Q_{r_0}}
    \big[ f(t,x,v) - 25r_0^2 \|S\|_{L^\infty(Q_1)}
    -(1-\delta_1)\big]_{+}\\
  & \geq \frac{1}{|Q_{r_0}|}
    \int_{\{f\geq 1-\theta\}\cap Q_{r_0}}
    \left(\frac{\delta_1}{2}-\theta\right)_{+} 
    \ge \delta_2 \left(\frac{\delta_1}{2}-\theta\right).
\end{align}

We then estimate from above the right hand side of the Poincar\'e
inequality \eqref{Poinc in proof}:
\begin{equation*}
  \int_{Q_{5r_0}} |\nabla_v g_+| \le \int_{Q_{5r_0}}
  |\nabla_v f_+|
  \le  \int_{\{ f=0 \} \cap Q_{5r_0}} \cdots
  +\int_{\{ 0<f<1-\theta \} \cap Q_{5r_0}} \cdots
  + \int_{\{ f\geq 1-\theta \} \cap Q_{5r_0}} \cdots
  =:  I_1 +I_2 + I_3.
\end{equation*}
The first term $I_1=0$ since $\nabla_v f_+ = 0$ almost everywhere
on $\{f_+=0\}$
(see~\cite[Subsection~4.2.2]{MR3409135}). Combining the
Cauchy-Schwarz inequality, Proposition \ref{prop:EE} and the fact
that $f\leq 1$, we get
\begin{equation*}
  \begin{array}{lll}
    I_2
    & \leq
    & |\{0<f<1-\theta\}\cap Q_{5r_0}|^{\frac12}
      \left(\displaystyle
      \fint_{Q_{5r_0}} |\nabla_v f_+|^2 \right)^{\frac12} \\
    & \lesssim
    & |\{0<f<1-\theta\}\cap Q_{\frac12}|^{\frac12}
      \left(\displaystyle\fint_{Q_{\frac12}} f_+^2 \right)^{\frac12}
      \lesssim  |\{0<f<1-\theta\}\cap Q_{\frac12}|^{\frac12}
  \end{array}
\end{equation*}
and (using that $\nabla_v f$ is zero almost everywhere on
$\{ f = \text{cst} \}$, see
again~\cite[Subsection~4.2.2]{MR3409135})
\begin{eqnarray*}
  I_3 &=
  &  \int_{Q_{5r_0}} \big|\nabla_v
    \big[(f-(1-\theta))_+ +(1-\theta) \big]\big|
    = \int_{Q_{5r_0}}
    \big|\nabla_v \big[f-(1-\theta)\big]_+\big| 
    \lesssim \left( \int_{Q_{5r_0}}
    \big|\nabla_v \big[f-(1-\theta)\big]_+\big|^2
    \right)^{\frac12} \\
      &\lesssim
  &  \left( \int_{Q_{\frac12}}
    \big[f-(1-\theta)\big]_+^2 
    + \int_{Q_{\frac12}}
    \big[f-(1-\theta)\big]_+ |S| \right)^{\frac12}
    \lesssim  \theta + \theta^{\frac12}\|S\|_{L^\infty(Q_1)}
    \lesssim \theta^{\frac12}\left(1+\|S\|_{L^\infty(Q_1)}\right)
\end{eqnarray*}
where we have used the energy estimate in
Proposition~\ref{prop:EE} on $[f-(1-\theta)]_+$.

The last two estimates on $I_2$ and $I_3$ yield the following
control on the right hand side of~\eqref{Poinc in proof}:
\begin{equation}
  \label{upp bd Poinc}
  \frac{1}{r_0^{4d+1} \varepsilon^{d+2}}
  \int_{Q_{5r_0}} |\nabla_v g_+| +
  \varepsilon^\sigma \lesssim
  \frac{\theta^{\frac12}\left(1+\|S\|_{L^\infty(Q_1)}\right)}{r_0^{4d+1}
    \varepsilon^{d+2}}
  + \frac{|\{0<f<1-\theta\}\cap
    Q_{\frac12}|^{\frac12}}{r_0^{4d+1}\varepsilon^{d+2}}
  + \varepsilon^\sigma.
\end{equation}

Combining \eqref{lower bd Poinc} and \eqref{upp bd Poinc} gives,
for some universal constant $C\geq 1$:
\begin{equation}
  \label{estim finale}
  \frac{\delta_1 \delta_2}{2} \le \delta_2 \theta +
  \frac{C\left(1+\|S\|_{L^\infty(Q_1)} \right)
    \theta^{\frac12}}{r_0^{4d+1} \varepsilon^{d+2}} 
  + \frac{C|\{0<f<1-\theta\}\cap
    Q_{\frac12}|^{\frac12}}{r_0^{4d+1} \var^{d+2}} + C \varepsilon^\sigma.
\end{equation}
We choose $\varepsilon$ such that
$C\varepsilon^\sigma \leq \frac{\delta_1 \delta_2}{8}$ and
$\theta$ such that
$\delta_2 \theta + \frac{C\left(1+\|S\|_{L^\infty(Q_1)}\right)
  \theta^{\frac12}}{r_0^{4d+1}\var^{d+2}} \leq \frac{\delta_1
  \delta_2}{8}$, e.g.
\begin{equation}
  \label{eq:choice-theta}
  \varepsilon=\left(\frac{\delta_1\delta_2}{8C}\right)^{\frac{1}{\sigma}},
  \qquad
  \theta= \delta_1^2 \delta_2 ^2 \left[ 8\left( \delta_2 +
        \frac{C\left(1+\|S\|_{L^\infty(Q_1)}\right)}{r_0^{4d+1} \left(
            \frac{\delta_1\delta_2}{8C}
          \right)^{\frac{d+2}{\sigma}}}\right) \right]^{-2},
\end{equation}
which finally implies the result with
\begin{equation}
  \label{eq:choice-nu}
  \nu:= \frac{1}{|Q_{\frac12}|}\left(\frac{\delta_1\delta_2}{4C}
    \left(\frac{\delta_1\delta_2}{8C}\right)^{\frac{d+2}{\sigma}}
    r_0^{4d+1}\right)^2 \gtrsim
  \frac{\left( \delta_1 \delta_2
    \right)^{10d+16}}{\left(1+\|S\|_{L^\infty(Q_1)}\right)^{4d+2}}.
\end{equation}

\subsection{Measure-to-pointwise estimate}
\label{ss:m2p}

In this subsection, we combine Proposition \ref{prop:1st lemma}
and Theorem \ref{t:IVL} to prove a measure-to-pointwise estimate
of ``lowering of the maximum'' \`a la De Giorgi.

\begin{lemma}[Measure-to-pointwise upper bound]
  \label{l:increase}
  Given $\delta \in (0,1)$, define
  $r_0 = (\frac{\delta}{800})^{\frac12}$ if $S$ non-zero and
  $r_0=\frac{1}{20}$ if $S=0$. There is
  $\mu:=\mu(\delta)\sim \delta^{2 (1+\delta^{-10d-16})}>0$ such
  that any sub-solution $f$ to~\eqref{e:main}-\eqref{e:hyp-coef}
  in $Q_1$ with $S$ such that
  $\| S\|_{L^{\infty}(Q_1)}\le \mu$ and so that $f\leq 1$ in
  $Q_{\frac{1}{2}}$ and
  \begin{align}
    \label{eq:measure-hyp-pointwise}
    \left| \{ f \leq 0 \} \cap Q_{r_0}^- \right| \ge \delta
    \left| Q_{r_0}^- \right|
  \end{align}
  satisfies $f \le 1-\mu$ in $Q_{\frac{r_0}{2}}$, with
  $Q_{r_0}^{-} := Q_{r_0}(-2r_0^2,0,0) =(-3r_0^2,-2r_0^2] \times
  B_{r_0^3} \times B_{r_0}$ (see Figure~\ref{fig:LVI}).
\end{lemma}

\begin{proof}
  In view of Proposition~\ref{prop:1st lemma} and the scaling
  invariance, there is $\delta'>0$ depending only on $\lambda$
  and $\Lambda$ such that for any $r>0$, any sub-solution $f$ on
  $Q_{2r}$ so that $\int_{Q_{r}} f_{+}^2 \leq \delta' |Q_{r}|$
  satisfies $f\leq \frac{1}{2}$ in $Q_{\frac{r}{2}}$ (imposing
  $C\mu\le \frac{1}{4}$ with $C$ the universal constant in the
  estimate of Proposition~\ref{prop:1st lemma} used here). Define
  then $\nu, \theta>0$ as
  in~\eqref{eq:choice-theta}-\eqref{eq:choice-nu} with
  $\delta_1=\delta$ and $\delta_2=\delta'$ and a source term
  bounded in $L^\infty$ by $1$.

  Define $f_k := \theta^{-k}[f-(1-\theta^k)]$ for $k \ge 0$. The
  functions $f_{k}$ are sub-solutions to
  \eqref{e:main}-\eqref{e:hyp-coef} for all $k \ge 0$ with a
  source term of $L^{\infty}$ norm less than $1$ as long as
  $k\leq 1+\frac{1}{\nu}$ (assuming
  $\| S\|_{L^{\infty}(Q_1)}\le \mu$ so that
  $\| S\|_{L^{\infty}(Q_1)} \le \theta^{1+\frac{1}{\nu}}$). The
  sets $\{ 0<f_{k}<1-\theta\}= \{1-\theta^k<f<1-\theta^{k+1}\}$
  are disjoints and each $f_k$ satisfies
  \eqref{eq:measure-hyp-pointwise}. If
  $\int_{Q_{r_0}} (f_k)_{+}^2 \leq \delta' |Q_{r_0}|$ then
  $f_k\leq \frac{1}{2}$ in $Q_{\frac{r_0}{2}}$ so $f\le 1-\mu$
  with $\mu = \frac{\theta^k}{2}$ which concludes the proof.
  Consider $ 1\leq k_0 \leq 1+\nu^{-1}$ such that
  $\int_{Q_{r_0}} (f_{k})_{+}^2 > \delta'|Q_{r_0}|$ for any
  $0\leq k\leq k_0$.  Then for $0\leq k\leq k_0-1$
  \begin{align*}
    & \left|\left\{ f_{k}\geq 1-\theta \right\} \cap
      Q_{r_0}\right| = \left|\{f_{k+1}\geq
      0\} \cap Q_{r_0} \right| \geq \int_{Q_{r_0} } (f_{k+1})_{+}^2
      > \delta' |Q_{r_0}| \\
    & \left|\left\{f_{k}\leq 0\right\} \cap Q_{r_0}^{-}\right|
      \geq \left|\left\{f\leq 0 \right\}\cap Q_{r_0}^{-}\right|
      \geq \delta |Q_{r_0}^-|.
  \end{align*}
  Theorem \ref{t:IVL} for sub-solutions with source term of norm
  $L^{\infty}$ less than $1$ then implies, choosing $r_0  =
  (\frac{\delta}{800})^{\frac12}$, 
  \begin{align*}
    \left| \left\{ 0<f_{k}< 1-\theta \right\}\cap Q_{\frac12}\right|
    \geq \nu |Q_{\frac12}|.
  \end{align*}
  Summing these estimates and using the fact that the sets are
  disjoints we have
  \begin{align*}
    |Q_{\frac12}|\geq \sum_{k=0}^{k_0-1}
    \left| \left\{ 0<f_{k}< 1-\theta \right\}\cap
    Q_{\frac12} \right| \geq k_0 \nu |Q_{\frac12}|.
  \end{align*}
  So $k_0 \leq \nu^{-1}$ which ensures that source terms remain
  indeed less than one along the iteration, and we deduce
  \begin{align*}
    f \leq 1-\frac{\theta^{k_0+1}}{2}\leq
    1-\frac{\theta^{\frac{1+\nu}{\nu}}}{2} \quad \mbox{ in }
    Q_{\frac{r_0}{2}}
  \end{align*}
  which yields
  $\mu(\delta) := \frac{\theta^{1+\frac{1}{\nu}}}{2} \sim
  \delta^{2 (1+\delta^{-10d-16})}$.
\end{proof}

\section{Harnack inequalities and H\"{o}lder continuity}
\label{sec:Harnack}

\subsection{The Harnack inequalities}

To prove the weak Harnack inequality, we first assume $S=0$, and
re-introduce $S$ in the end. Without source term,
$r_0=\frac{1}{20}$ can be taken constant in the
measure-to-pointwise estimate. Consider then $h$ non-negative
super-solution to~\eqref{e:main}-\eqref{e:hyp-coef} on $Q_1$ with
$S=0$. The contraposition of Lemma~\ref{l:increase} on the
sub-solution $g := 1 - \frac{h}{M}$ then implies for any
$\delta \in (0,1)$ and $M \sim \delta^{-2 (1+\delta^{-10d-16})}$
that
\begin{equation}
  \label{eq:scaling-invariant}
  \forall \, Q_r(z) \subset Q_1 \text{ with }
  Q_{\frac{r}{2}}^{+}(z) \subset Q_1, \quad
  \frac{\left| \{ h > M \} \cap Q_{r}(z) \right|}{\left|
  Q_{r}(z) \right|}  >  \delta \quad
  \Longrightarrow \quad \inf_{Q_{\frac{r}{2}}^+(z)} h \ge 1
\end{equation}
where
$Q_{\frac{r}{2}}^{+}(z) = Q_{\frac{r}{2}}(z +(2r^2,2r^2v,0))$,
for $z=(t,x,v)$, is obtained by inverting the operation
$Q_{\frac{r}{2}}(z) \to Q_{r}^-(z)$ in Lemma~\ref{l:increase}
(noting that $Q_r^-(z)=Q_r(z-(2r^2,2r^2v,0))$). It implies
(inverting the relation $\delta \to M$ and using the layer-cake
representation) that if $\inf_{Q_{\frac{r_0}{2}}} h <1$,
\begin{align}
  \label{eq:initial}
  \forall \, M \ge 1, \quad \frac{\left| \{ h \ge M \} \cap
  Q_{r_0}^-  \right|}{\left|Q_{r_0}^-  \right|}  \lesssim
  \delta(M) =\left( \frac{1}{ \ln (1+M)}
  \right)^{\frac{1}{10d+17}} \quad \Longrightarrow \quad
  \int_{Q_{r_0}^- } \left[ \ln \left( 1+ h \right)
  \right]^{\frac{1}{10d+18}} \lesssim 1.
\end{align}
This ``point-to-measure'' estimate controls the decay of the
upper level set in the manner of a \emph{weak Harnack
  inequality}, although with a ``logarithmic'' rather than
power-law integrability. We shall now improve the integrability
to a power-law by going back to~\eqref{eq:scaling-invariant} and
performing an inductive argument inspired from the elliptic
theory~\cite{MR3565366}. Note that the logarithmic integrability
in~\eqref{eq:initial} is reminiscent of Moser's approach.

We improve inductively the control of upper level sets in the
following decreasing sequence of cylinders
\begin{equation*}
  \cQ^k:= Q_{\frac{r_0}{2}+\alpha_k} \left(-\frac{5}{2} r_0^2 +
    \frac12 \left( \frac{r_0}{2}+\alpha_k \right)^2,0,0 \right)\subset
  Q_{r_0}^{-} \quad \text{ with } \quad
  \alpha_k:=\frac{r_0}{2\times 7^{k-1}}.
\end{equation*}
These cylinders satisfy
$\tilde Q^-_{\frac{r_0}{2}} \subset \cQ^k \subset \bar \cQ^k
\subset \mathring{\cQ}^{k-1} \subset Q^- _{r_0}$ for all
$k \ge 1$. We now claim that for $\delta_0>0$ small enough (to be
chosen later), for any non-negative super-solution $h$ with
$\inf_{Q_{\frac{r_0}{2}}} h <1$ we have
\begin{align}
  \label{eq:claim}
  \forall \, k \ge 1, \quad
  \frac{\left| \{ h \ge M^k \} \cap \cQ^k \right|}{\left|
  \cQ^k  \right|}  \le \frac{\delta_0}{210^{(4d+2)k}}
\end{align}
where $M \sim \delta^{-2 (1+\delta^{-10d-16})}$ with
$\delta:=\frac{\delta_0}{210^{4d+2}}$ as
in~\eqref{eq:scaling-invariant}. Admitting first~\eqref{eq:claim}
we deduce by layer-cake representation that
there is an explicit $\zeta \gtrsim \delta_0 ^{10d+17}>0$ such
that $\int_{\tilde Q_{\frac{r_0}{2}}^{-}} h^\zeta \dd z \lesssim
1$, which implies by linearity
\begin{align*}
  \left(  \int_{\tilde Q_{\frac{r_0}{2}}^{-}} h(z)^\zeta \dd z
  \right)^{\frac{1}{\zeta}} \lesssim \inf_{Q_{\frac{r_0}{2}}} h. 
\end{align*}
This implies the weak Harnack
inequality~\eqref{eq:w-Harnack-stat} on any $f$ non-negative
super-solution to~\eqref{e:main}-\eqref{e:hyp-coef} by applying
the previous estimate to $h := f + (1+t) \|
S\|_{L^\infty(Q_1)}$. To deduce the Harnack
inequality~\eqref{eq:s-Harnack-stat} we consider $f$ a
non-negative solution to~\eqref{e:main}-\eqref{e:hyp-coef} and
combine the previous control with Proposition~\ref{prop:1st
  lemma} to get
\begin{align*}
  \sup_{\tilde Q_{\frac{r_0}{4}} ^{-}} f \lesssim
  \left( \int_{\tilde Q_{\frac{r_0}{2}} ^{-}} f(z)^\zeta \dd z
  \right)^{\frac{1}{\zeta}} +  \| S\|_{L^\infty(Q_1)} \lesssim
  \inf_{Q_{\frac{r_0}{2}}} f  +  \| S\|_{L^\infty(Q_1)} \lesssim
  \inf_{Q_{\frac{r_0}{4}}} f  +  \| S\|_{L^\infty(Q_1)}.
\end{align*}

Let us now prove the claim~\eqref{eq:claim} to conclude the
proof. The initialization $k=1$ is proved
in~\eqref{eq:initial}. Then define
$A_{k+1} := \{ h > M^{k+1} \} \cap \cQ^{k+1}$ and denote
the following translated centered cylinders
$\mathfrak C_{r}[z] :=z\circ Q_{2r}((2r^2,0,0))= z \circ
\left(-2r^2,2r^2\right]\times B_{(2r)^3} \times B_{2r}$.  Let us
construct $z_\ell=(t_\ell,x_\ell,v_\ell)\in \cQ^{k+1}$ and
$r_\ell >0$, $\ell \ge 1$, so that:
\begin{enumerate}
\item $\forall \, \ell \ge 1$, $r_\ell \in (0,\frac{\alpha_{k+1}}{15})$,
\item $\forall \, \ell \ge 1$,
  $|A_{k+1} \cap \mathfrak C_{15r_\ell}[z_\ell]| \le \delta_0
  |\mathfrak C_{15 r_\ell}[z_\ell]|$,
\item $\forall \, \ell \ge 1$,
  $|A_{k+1} \cap \mathfrak C_{r_\ell}[z_\ell]| > \delta_0
  |\mathfrak C_{r_\ell}[z_\ell]|$,
\item the cylinders $\mathfrak C_{3r_\ell}[z_\ell]$, $\ell \ge
  1$, are disjoint,
\item $A_{k+1}$ is covered by the family $\mathfrak
  C_{15r_\ell}[z_\ell]$, $\ell \ge 1$.
\end{enumerate}
Note that inverting the operation
$Q_{\frac{r}{2}}(z) \to Q_{r}^-(z)$ in Lemma~\ref{l:increase}
yields, when starting from $\mathfrak C_{r_\ell}[z_\ell]$, the
cylinder
$\mathfrak C_{r_\ell}[z_\ell]^+ :=z_\ell\circ
Q_{r_\ell}((10r_\ell^2,0,0))= z_\ell \circ
\left(9r_\ell^2,10r_\ell^2\right]\times B_{r_\ell^3} \times
B_{r_\ell}$. Note also that
$\mathfrak C_{r_\ell}[z_\ell]^+ \subset \mathfrak C_{3
  r_\ell}[z_\ell]$ and that property (1) combined with
$z_\ell \in \cQ^{k+1}$ imply
$\mathfrak C_{15r_\ell}[z_\ell] \subset \cQ^k$.  Let us prove
that the family $\mathcal F$ of cylinders $\mathfrak C_{r}[z]$
with $z \in \cQ^{k+1}$, $r \in (0,\frac{\alpha_{k+1}}{15})$ and
so that
$|A_{k+1} \cap \mathfrak C_{15r}[z]| \le \delta_0 |\mathfrak
C_{15 r}[z]|$ and
$|A_{k+1} \cap \mathfrak C_{r}[z]| > \delta_0 |\mathfrak
C_{r}[z]|$ cover $A_{k+1}$. We have, using~\eqref{eq:claim} at
the previous step $k$,
\begin{equation}
  \label{eq:constraint-varphi}
  \forall \, r \in \left(\frac{\alpha_{k+1}}{15},\alpha_{k+1}\right), \quad
  |A_{k+1} \cap \mathfrak C_{r}[z]| \le |A_{k} \cap \mathfrak
  C_{r}[z]| \le |A_{k} \cap \cQ^k| \le \frac{\delta_0}{210^{(4d+2)k}}
  |\cQ^k| \le \delta_0 |\mathfrak C_{r}[z]|.
\end{equation}
If $z \in A_{k+1}$ is not covered by $\mathcal F$ it means that
the continuous positive function
$\varphi(r) = \frac{|A_{k+1} \cap \mathfrak C_{r}[z]|}{|\mathfrak
  C_{r}[z]|}$ on $(0,+\infty)$ satisfies $\varphi(r) \le \delta_0$
or $\varphi(15r) > \delta_0$ for all
$r \in (0,\frac{\alpha_{k+1}}{15})$. The
constraint~\eqref{eq:constraint-varphi} and the continuity impose
$\varphi(r) \le \delta_0$ for all
$r \in (0,\frac{\alpha_{k+1}}{15})$. Taking $r \to 0$, a
straightforward variation of the Lebesgue differentiation theorem
then implies $z\notin A_{k+1}$ which contradicts the
assumption. Hence $A_{k+1}$ is covered by the family $\mathcal
F$. 

It implies in particular that $A_{k+1}$ is covered by the family
$\mathcal F'$ of cylinders $\mathfrak C_{3r}[z]$ with $z \in \cQ^{k+1}$,
$r \in (0,\frac{\alpha_{k+1}}{15})$ and such that
$|A_{k+1} \cap \mathfrak C_{15r}[z]| \le \delta_0 |\mathfrak C_{15
  r}[z]|$ and
$|A_{k+1} \cap \mathfrak C_{r}[z]| > \delta_0 |\mathfrak
C_{r}[z]|$. The Vitali covering lemma then gives the existence of
a countable sub-family, denoted
$(\mathfrak C_{r_\ell}[z_\ell])_{\ell \ge 1}$, such that the
$(\mathfrak C_{15 r_\ell}[z_\ell])_{\ell \ge 1}$ cover $A_{k+1}$
and the $(\mathfrak C_{3 r_\ell}[z_\ell])_{\ell \ge 1}$ are
disjoint. The Vitali lemma applies thanks to the following
property:
\begin{equation*}
  \Big[ \mathfrak C_{r_1}[z_1] \cap \mathfrak C_{r_2}[z_2] \neq
  \varnothing \ \mbox{ and } \ r_1\le 2r_2 \Big] \quad
  \Longrightarrow \quad \mathfrak C_{r_1}[z_1] \subset
  \mathfrak C_{5r_2}[z_2].
\end{equation*} 
Take $z_0=(t_0,x_0,v_0)$ in the intersection and
$z=(t,x,v)\in \mathfrak C_{r_1}[z_1]$. Inequalities
$|t-t_2|\leq 18 r_2^2$ and $|v-v_2|\leq 10r_2$ come naturally and
$|x-[x_2+2 r_2^2v_2 +(t-t_2)v_2]| \le 200r_2^3$ follows from
\begin{align*}
  & \big| x- \big[ x_2+ 2r_2^2v_2 + (t-t_2)v_2 \big] \big| \\
  &\le   \big| x- \big[ x_1+ 2r_1^2v_1 + (t-t_1)v_1 \big] \big| +
   \big| \big[ x_2+ 2r_2^2v_2 + (t-t_2)v_2 \big] - \big[ x_1+
    2r_1^2v_1 + (t-t_1)v_1 \big] \big| \\
  & \le r_1^3 + \big| \big[ x_2+ 2r_2^2v_2 + (t_0-t_2)v_2 \big] -
    \big[ x_1+ 2r_1^2v_1 + (t_0-t_1)v_1 \big] \big| + \big|
    (t-t_0)(v_2-v_1) \big| \\
  & \le 128 r_2 ^3 + \big| x_0- \big[ x_1+ 2r_1^2v_1 +
    (t_0-t_1)v_1 \big] \big| + \big| x_0- \big[ x_2+ 2r_2^2v_2 +
    (t_0-t_2)v_2 \big] \big|\le 200 r_2 ^3.
\end{align*}
This finishes constructing the covering with the properties
(1)-(2)-(3)-(4)-(5) above. Then Lemma~\ref{l:increase} applied to
each $\mathfrak C_{r_\ell}[z_\ell]$ implies
$\mathfrak C_{r_\ell}[z_\ell]^+ \subset A_k$, and the
$\mathfrak C_{r_\ell}[z_\ell]^+ \subset \mathfrak C_{3
  r_\ell}[z_\ell]$ are disjoint.  We deduce
\begin{align*}
  |A_{k+1}|
  & \le \sum_{\ell \ge 1} |A_{k+1} \cap \mathfrak C_{15
    r_\ell}[z_\ell]| \le \delta_0 \sum_{\ell \ge 1}
    |\mathfrak C_{15 r_\ell}[z_\ell]| 
   \le 15^{4d+2} \delta_0 \sum_{\ell \ge 1} |\mathfrak
    C_{r_\ell}[z_\ell]| \le 30^{4d+2} \delta_0 \sum_{\ell \ge 1}
    |\mathfrak C_{r_\ell}[z_\ell]^+| \\
  & \le 30^{4d+2} \delta_0 |A_k|
    \le \frac{30^{4d+2}\delta_0^2}{210^{(4d+2)k}}
    \le \frac{\delta_0}{210^{(4d+2)(k+1)}} \left| \cQ^{k+1}
    \right|
\end{align*}
for $\delta_0$ small enough which proves the induction
claim~\eqref{eq:claim} and concludes the proof.

\subsection{The H\"{o}lder continuity}
\label{ss:Holder}

De Giorgi's argument to H\"{o}lder continuity uses the
measure-to-pointwise Lemma~\ref{l:increase}. We briefly sketch it
in order to track the constant. H\"{o}lder regularity could also be
deduced from the Harnack inequality in
Theorem~\ref{t:harnack}. Given $f$ solution to
\eqref{e:main}-\eqref{e:hyp-coef} on $Q_2$ and
$r_0 =\frac{1}{40}$
\begin{equation}
  \label{eq:reduc-osc}
  \underset{Q_{r_0}}{\mathrm{osc}}
  \mbox{ }f \leq \left( 1- \frac{\mu}{2} \right)
  \mbox{ }\max\left(\underset{Q_{1}}{\mathrm{osc}}
    \mbox{ }f,e^{2(1+2^{10d+16})} \|S\|_{L^{\infty}(Q_2)}\right)
\end{equation}
follows from~Lemma \ref{l:increase} rescaled to $Q_2$ with
$\delta=\frac12$ and applied to whichever of $F$ or $-F$
satisfies~\eqref{eq:measure-hyp-pointwise}, where
$F : =2 [\max(\mathrm{osc}_{Q_{1}} f,
    e^{2(1+2^{10d+16})}
    \|S\|_{L^{\infty}(Q_2)})]^{-1} [f-\frac12
  (\sup_{Q_{1}} f+\inf_{Q_{1}} f)]$. By iteration we deduce
\begin{equation}
  \label{eq:iter-osc}
  \forall z_0 \in Q_{1}, \forall \, r\in \left(0,r_0\right),
  \quad
  \underset{Q_{r}(z_0)}{\mathrm{osc}} \mbox{ }f 
   \leq r^{\alpha} e^{2(1+2^{10d+16})}
    \left(1+\|S\|_{L^{\infty}(Q_2)}\right)\max\left(e^{2(1+2^{10d+16})}
      \|S\|_{L^{\infty}(Q_2)},\mbox{ }\underset{Q_1}{\mathrm{osc}}
    \mbox{ }f \right).
\end{equation}
Indeed, the following sequence of solution of
\eqref{e:main}-\eqref{e:hyp-coef} in $Q_1$
\begin{align*}
  f_{n}(\tau,y,w)= 2 
  \frac{\left( 1 - \frac{\mu}{2}
  \right)^{1-n}}{\max\left(\underset{Q_{1}}{\mathrm{osc}}
  \mbox{ }f, e^{2(1+2^{10d+16})}
  \|S\|_{L^{\infty}(Q_2)}\right)}  \,
  f\left(t_0+ r_0 ^{2n} \tau, x_0
  -r_0 ^{2n} \tau v_0 + r_0 ^{3n} y,v_0+ r_0
  ^n w \right).
\end{align*}
satisfies
$\underset{Q_{1}}{\mathrm{osc}} \mbox{ } f_{n} \leq 2
e^{2(1+2^{10d+16})}(1+\|S\|_{L^{\infty}(Q_2)})$ by induction on
$n \ge 1$ (the case $n=1$ is true by definition of $f_n$ and it
propagates thanks to~\eqref{eq:reduc-osc}). If one defines
$\alpha\in (0,1)$ such that $1-\frac{\mu}{2}=r_0^{\alpha}$
(assuming that $\mu$ is small enough), the previous induction
implies~\eqref{eq:iter-osc} by standard arguments. To deduce the
H\"{o}lder estimate between $z,z'\in Q_1$, use intermediate
points in $[z,z']$ at distance less than $r_0$ and the
estimate~\eqref{eq:iter-osc}.

\bigskip
\noindent {\bf Acknowledgements.} The authors are grateful to
C.~Imbert for the inspirational interactions, as well as for
specific help with the litterature and the comparison with the
Moser-Kru\v{z}kov approach in \cite{Guerand-Imbert}. The second
author would also like to thank L.~Silvestre who pointed out
several years ago how Kolmogorov fundamental solutions were used
in~\cite{pp} to replace averaging lemma, which was the starting
point of our Section~\ref{sec:int} (and is also used
in~\cite{MR4049224}). The authors acknowledge funding by
the ERC grant MAFRAN 2017-2022.

\bibliographystyle{amsalpha}
\bibliography{Gehring}

\end{document}